\RequirePackage{fix-cm}
\documentclass[smallextended,envcountsect]{svjour3} 
\smartqed  
\usepackage{graphicx}
\usepackage{algorithm}
\usepackage{algpseudocode}
\usepackage{amsmath}
\usepackage{amssymb}
\usepackage{mathtools}
\usepackage{dcolumn}
\usepackage{subcaption}
\usepackage{bm}
\usepackage{graphbox}
\usepackage{tabularray}
\usepackage{wrapfig}
\usepackage{amsfonts}
\usepackage{multirow}
\usepackage{mdframed}
\usepackage{array}
\usepackage{tikz}
\usepackage{hyperref}
\usepackage[nameinlink]{cleveref}
\usetikzlibrary{graphs}
\crefname{mp}{problem}{problems}

\newcommand{\x}{\mathbf{x}}
\newcommand{\p}{\mathbf{p}}
\newcommand{\w}{\mathbf{w}}

\newcommand{\R}[1]{\mathbb{R}^{#1}}
\newcommand{\N}{\mathcal{N}}
\newcommand{\J}{\mathcal{J}}

\newcommand{\interior}[1]{\mathrm{Int} #1}
\newcommand{\closed}[1]{\mkern 1.5mu\overline{\mkern-1.5mu#1\mkern-1.5mu}\mkern 1.5mu}

\DeclareMathOperator*{\argmin}{\arg\!\min}
\begin{document}

\title{Mathematical Program Networks
}


\author{Forrest Laine
}


\institute{Forrest Laine \at
           Computer Science Department \\
           Vanderbilt University \\
           Nashville, TN, USA \\
              \email{forrest.laine@vanderbilt.edu}           
}

\date{April 23, 2024}

\maketitle
\begin{abstract}

Mathematical Program Networks (MPNs) are introduced in this work. An MPN is a collection of interdependent Mathematical Programs (MPs) which are to be solved simultaneously, while respecting the connectivity pattern of the network defining their relationships. The network structure of an MPN impacts which decision variables each constituent mathematical program can influence, either directly or indirectly via solution graph constraints representing optimal decisions for their decedents. Many existing problem formulations can be formulated as MPNs, including Nash Equilibrium problems, multilevel optimization problems, and Equilibrium Programs with Equilibrium Constraints (EPECs), among others. The equilibrium points of an MPN correspond with the equilibrium points or solutions of these other problems. By thinking of a collection of decision problems as an MPN, a common definition of equilibrium can be used regardless of relationship between problems, and the same algorithms can be used to compute solutions. The presented framework facilitates modeling flexibility and analysis of various equilibrium points in problems involving multiple mathematical programs. 

\keywords{Multilevel \and Equilibrium \and Mathematical Programming \and Game Theory}
\end{abstract}

\newpage

\expandafter\newcommand\csname graph1\endcsname{%
\begin{tikzpicture}[scale=0.5,baseline=(current bounding box.center)]
\node[circle,draw,fill=yellow!70] (1) at (-1.5,0.0) {};
\node[circle, draw] (2) at (-0.5,0.0) {};
\node[circle, draw] (3) at (0.5,0.0) {};
\node[circle, draw] (4) at (1.5,0.0) {};
\graph {};
\end{tikzpicture}
}
\expandafter\newcommand\csname graph2\endcsname{%
\begin{tikzpicture}[scale=0.5,baseline=(current bounding box.center)]
\node[circle,draw,fill=yellow!70] (1) at (-1.0,0.0) {};
\node[circle, draw] (2) at (0.0,-1.0) {};
\node[circle, draw] (3) at (0.0,0.0) {};
\node[circle, draw] (4) at (1.0,0.0) {};
\graph {(1) -> (2); };
\end{tikzpicture}
}
\expandafter\newcommand\csname graph3\endcsname{%
\begin{tikzpicture}[scale=0.5,baseline=(current bounding box.center)]
\node[circle,draw,fill=yellow!70] (1) at (-1.0,0.0) {};
\node[circle, draw] (2) at (0.0,0.0) {};
\node[circle, draw] (3) at (0.0,-1.0) {};
\node[circle, draw] (4) at (1.0,0.0) {};
\graph {(2) -> (3); };
\end{tikzpicture}
}
\expandafter\newcommand\csname graph4\endcsname{%
\begin{tikzpicture}[scale=0.5,baseline=(current bounding box.center)]
\node[circle,draw,fill=yellow!70] (1) at (0.0,-1.0) {};
\node[circle, draw] (2) at (-1.0,0.0) {};
\node[circle, draw] (3) at (0.0,0.0) {};
\node[circle, draw] (4) at (1.0,0.0) {};
\graph {(2) -> (1); };
\end{tikzpicture}
}
\expandafter\newcommand\csname graph5\endcsname{%
\begin{tikzpicture}[scale=0.5,baseline=(current bounding box.center)]
\node[circle,draw,fill=yellow!70] (1) at (-0.5,0.0) {};
\node[circle, draw] (2) at (-0.5,-1.0) {};
\node[circle, draw] (3) at (0.5,-1.0) {};
\node[circle, draw] (4) at (0.5,0.0) {};
\graph {(1) -> (2); (1) -> (3); };
\end{tikzpicture}
}
\expandafter\newcommand\csname graph6\endcsname{%
\begin{tikzpicture}[scale=0.5,baseline=(current bounding box.center)]
\node[circle,draw,fill=yellow!70] (1) at (-0.0,-0.5) {};
\node[circle, draw] (2) at (1.0,0.0) {};
\node[circle, draw] (3) at (2.0,0.0) {};
\node[circle, draw] (4) at (-0.0,0.5) {};
\graph {(1) -> (2); (2) -> (3); };
\end{tikzpicture}
}
\expandafter\newcommand\csname graph7\endcsname{%
\begin{tikzpicture}[scale=0.5,baseline=(current bounding box.center)]
\node[circle,draw,fill=yellow!70] (1) at (1.0,0.0) {};
\node[circle, draw] (2) at (2.0,0.0) {};
\node[circle, draw] (3) at (-0.0,-0.5) {};
\node[circle, draw] (4) at (-0.0,0.5) {};
\graph {(1) -> (2); (3) -> (1); };
\end{tikzpicture}
}
\expandafter\newcommand\csname graph8\endcsname{%
\begin{tikzpicture}[scale=0.5,baseline=(current bounding box.center)]
\node[circle,draw,fill=yellow!70] (1) at (-1.0,0.0) {};
\node[circle, draw] (2) at (0.0,-1.0) {};
\node[circle, draw] (3) at (0.0,0.0) {};
\node[circle, draw] (4) at (1.0,0.0) {};
\graph {(3) -> (2); (1) -> (2); };
\end{tikzpicture}
}
\expandafter\newcommand\csname graph9\endcsname{%
\begin{tikzpicture}[scale=0.5,baseline=(current bounding box.center)]
\node[circle,draw,fill=yellow!70] (1) at (-0.5,0.0) {};
\node[circle, draw] (2) at (-0.5,-1.0) {};
\node[circle, draw] (3) at (0.5,0.0) {};
\node[circle, draw] (4) at (0.5,-1.0) {};
\graph {(1) -> (2); (3) -> (4); };
\end{tikzpicture}
}
\expandafter\newcommand\csname graph10\endcsname{%
\begin{tikzpicture}[scale=0.5,baseline=(current bounding box.center)]
\node[circle,draw,fill=yellow!70] (1) at (-0.5,0.0) {};
\node[circle, draw] (2) at (0.5,0.0) {};
\node[circle, draw] (3) at (-0.5,-1.0) {};
\node[circle, draw] (4) at (0.5,-1.0) {};
\graph {(2) -> (4); (2) -> (3); };
\end{tikzpicture}
}
\expandafter\newcommand\csname graph11\endcsname{%
\begin{tikzpicture}[scale=0.5,baseline=(current bounding box.center)]
\node[circle,draw,fill=yellow!70] (1) at (0.5,-1.0) {};
\node[circle, draw] (2) at (-0.5,0.0) {};
\node[circle, draw] (3) at (-0.5,-1.0) {};
\node[circle, draw] (4) at (0.5,0.0) {};
\graph {(2) -> (1); (2) -> (3); };
\end{tikzpicture}
}
\expandafter\newcommand\csname graph12\endcsname{%
\begin{tikzpicture}[scale=0.5,baseline=(current bounding box.center)]
\node[circle,draw,fill=yellow!70] (1) at (2.0,0.0) {};
\node[circle, draw] (2) at (-0.0,-0.5) {};
\node[circle, draw] (3) at (1.0,0.0) {};
\node[circle, draw] (4) at (-0.0,0.5) {};
\graph {(3) -> (1); (2) -> (3); };
\end{tikzpicture}
}
\expandafter\newcommand\csname graph13\endcsname{%
\begin{tikzpicture}[scale=0.5,baseline=(current bounding box.center)]
\node[circle,draw,fill=yellow!70] (1) at (-0.0,-0.5) {};
\node[circle, draw] (2) at (-0.0,0.5) {};
\node[circle, draw] (3) at (1.0,0.0) {};
\node[circle, draw] (4) at (2.0,0.0) {};
\graph {(3) -> (4); (2) -> (3); };
\end{tikzpicture}
}
\expandafter\newcommand\csname graph14\endcsname{%
\begin{tikzpicture}[scale=0.5,baseline=(current bounding box.center)]
\node[circle,draw,fill=yellow!70] (1) at (0.5,-1.0) {};
\node[circle, draw] (2) at (-0.5,0.0) {};
\node[circle, draw] (3) at (-0.5,-1.0) {};
\node[circle, draw] (4) at (0.5,0.0) {};
\graph {(4) -> (1); (2) -> (3); };
\end{tikzpicture}
}
\expandafter\newcommand\csname graph15\endcsname{%
\begin{tikzpicture}[scale=0.5,baseline=(current bounding box.center)]
\node[circle,draw,fill=yellow!70] (1) at (-1.0,0.0) {};
\node[circle, draw] (2) at (0.0,0.0) {};
\node[circle, draw] (3) at (0.0,-1.0) {};
\node[circle, draw] (4) at (1.0,0.0) {};
\graph {(4) -> (3); (2) -> (3); };
\end{tikzpicture}
}
\expandafter\newcommand\csname graph16\endcsname{%
\begin{tikzpicture}[scale=0.5,baseline=(current bounding box.center)]
\node[circle,draw,fill=yellow!70] (1) at (0.0,-1.0) {};
\node[circle, draw] (2) at (-1.0,0.0) {};
\node[circle, draw] (3) at (0.0,0.0) {};
\node[circle, draw] (4) at (1.0,0.0) {};
\graph {(3) -> (1); (2) -> (1); };
\end{tikzpicture}
}
\expandafter\newcommand\csname graph17\endcsname{%
\begin{tikzpicture}[scale=0.5,baseline=(current bounding box.center)]
\node[circle,draw,fill=yellow!70] (1) at (0.0,0.0) {};
\node[circle, draw] (2) at (0.0,-1.0) {};
\node[circle, draw] (3) at (1.0,-1.0) {};
\node[circle, draw] (4) at (-1.0,-1.0) {};
\graph {(1) -> (2); (1) -> (3); (1) -> (4); };
\end{tikzpicture}
}
\expandafter\newcommand\csname graph18\endcsname{%
\begin{tikzpicture}[scale=0.5,baseline=(current bounding box.center)]
\node[circle,draw,fill=yellow!70] (1) at (-0.0,0.0) {};
\node[circle, draw] (2) at (1.0,-0.5) {};
\node[circle, draw] (3) at (1.0,0.5) {};
\node[circle, draw] (4) at (2.0,0.0) {};
\graph {(2) -> (4); (1) -> (2); (1) -> (3); };
\end{tikzpicture}
}
\expandafter\newcommand\csname graph19\endcsname{%
\begin{tikzpicture}[scale=0.5,baseline=(current bounding box.center)]
\node[circle,draw,fill=yellow!70] (1) at (1.0,0.0) {};
\node[circle, draw] (2) at (2.0,-0.5) {};
\node[circle, draw] (3) at (2.0,0.5) {};
\node[circle, draw] (4) at (-0.0,0.0) {};
\graph {(1) -> (2); (4) -> (1); (1) -> (3); };
\end{tikzpicture}
}
\expandafter\newcommand\csname graph20\endcsname{%
\begin{tikzpicture}[scale=0.5,baseline=(current bounding box.center)]
\node[circle,draw,fill=yellow!70] (1) at (-0.5,0.0) {};
\node[circle, draw] (2) at (-0.5,-1.0) {};
\node[circle, draw] (3) at (0.5,-1.0) {};
\node[circle, draw] (4) at (0.5,0.0) {};
\graph {(1) -> (2); (4) -> (2); (1) -> (3); };
\end{tikzpicture}
}
\expandafter\newcommand\csname graph21\endcsname{%
\begin{tikzpicture}[scale=0.5,baseline=(current bounding box.center)]
\node[circle,draw,fill=yellow!70] (1) at (-0.0,0.0) {};
\node[circle, draw] (2) at (1.0,0.0) {};
\node[circle, draw] (3) at (2.0,0.5) {};
\node[circle, draw] (4) at (2.0,-0.5) {};
\graph {(2) -> (4); (1) -> (2); (2) -> (3); };
\end{tikzpicture}
}
\expandafter\newcommand\csname graph22\endcsname{%
\begin{tikzpicture}[scale=0.5,baseline=(current bounding box.center)]
\node[circle,draw,fill=yellow!70] (1) at (-0.0,0.0) {};
\node[circle, draw] (2) at (1.0,0.0) {};
\node[circle, draw] (3) at (2.0,0.0) {};
\node[circle, draw] (4) at (3.0,0.0) {};
\graph {(1) -> (2); (3) -> (4); (2) -> (3); };
\end{tikzpicture}
}
\expandafter\newcommand\csname graph23\endcsname{%
\begin{tikzpicture}[scale=0.5,baseline=(current bounding box.center)]
\node[circle,draw,fill=yellow!70] (1) at (1.0,0.0) {};
\node[circle, draw] (2) at (2.0,0.0) {};
\node[circle, draw] (3) at (3.0,0.0) {};
\node[circle, draw] (4) at (-0.0,0.0) {};
\graph {(1) -> (2); (4) -> (1); (2) -> (3); };
\end{tikzpicture}
}
\expandafter\newcommand\csname graph24\endcsname{%
\begin{tikzpicture}[scale=0.5,baseline=(current bounding box.center)]
\node[circle,draw,fill=yellow!70] (1) at (-0.0,-0.5) {};
\node[circle, draw] (2) at (1.0,0.0) {};
\node[circle, draw] (3) at (2.0,0.0) {};
\node[circle, draw] (4) at (-0.0,0.5) {};
\graph {(1) -> (2); (4) -> (2); (2) -> (3); };
\end{tikzpicture}
}
\expandafter\newcommand\csname graph25\endcsname{%
\begin{tikzpicture}[scale=0.5,baseline=(current bounding box.center)]
\node[circle,draw,fill=yellow!70] (1) at (-0.0,0.0) {};
\node[circle, draw] (2) at (1.0,0.0) {};
\node[circle, draw] (3) at (2.0,0.0) {};
\node[circle, draw] (4) at (-0.0,1.0) {};
\graph {(1) -> (2); (4) -> (3); (2) -> (3); };
\end{tikzpicture}
}
\expandafter\newcommand\csname graph26\endcsname{%
\begin{tikzpicture}[scale=0.5,baseline=(current bounding box.center)]
\node[circle,draw,fill=yellow!70] (1) at (1.0,0.5) {};
\node[circle, draw] (2) at (2.0,0.0) {};
\node[circle, draw] (3) at (-0.0,0.0) {};
\node[circle, draw] (4) at (1.0,-0.5) {};
\graph {(1) -> (2); (3) -> (1); (3) -> (4); };
\end{tikzpicture}
}
\expandafter\newcommand\csname graph27\endcsname{%
\begin{tikzpicture}[scale=0.5,baseline=(current bounding box.center)]
\node[circle,draw,fill=yellow!70] (1) at (1.0,0.0) {};
\node[circle, draw] (2) at (2.0,0.0) {};
\node[circle, draw] (3) at (-0.0,-0.5) {};
\node[circle, draw] (4) at (-0.0,0.5) {};
\graph {(1) -> (2); (3) -> (1); (4) -> (1); };
\end{tikzpicture}
}
\expandafter\newcommand\csname graph28\endcsname{%
\begin{tikzpicture}[scale=0.5,baseline=(current bounding box.center)]
\node[circle,draw,fill=yellow!70] (1) at (1.0,0.0) {};
\node[circle, draw] (2) at (2.0,0.0) {};
\node[circle, draw] (3) at (-0.0,0.0) {};
\node[circle, draw] (4) at (-0.0,1.0) {};
\graph {(1) -> (2); (3) -> (1); (4) -> (2); };
\end{tikzpicture}
}
\expandafter\newcommand\csname graph29\endcsname{%
\begin{tikzpicture}[scale=0.5,baseline=(current bounding box.center)]
\node[circle,draw,fill=yellow!70] (1) at (2.0,0.0) {};
\node[circle, draw] (2) at (3.0,0.0) {};
\node[circle, draw] (3) at (1.0,0.0) {};
\node[circle, draw] (4) at (-0.0,0.0) {};
\graph {(1) -> (2); (3) -> (1); (4) -> (3); };
\end{tikzpicture}
}
\expandafter\newcommand\csname graph30\endcsname{%
\begin{tikzpicture}[scale=0.5,baseline=(current bounding box.center)]
\node[circle,draw,fill=yellow!70] (1) at (-0.5,0.0) {};
\node[circle, draw] (2) at (-0.5,-1.0) {};
\node[circle, draw] (3) at (0.5,0.0) {};
\node[circle, draw] (4) at (0.5,-1.0) {};
\graph {(3) -> (2); (1) -> (2); (3) -> (4); };
\end{tikzpicture}
}
\expandafter\newcommand\csname graph31\endcsname{%
\begin{tikzpicture}[scale=0.5,baseline=(current bounding box.center)]
\node[circle,draw,fill=yellow!70] (1) at (-1.0,0.0) {};
\node[circle, draw] (2) at (0.0,-1.0) {};
\node[circle, draw] (3) at (0.0,0.0) {};
\node[circle, draw] (4) at (1.0,0.0) {};
\graph {(3) -> (2); (1) -> (2); (4) -> (2); };
\end{tikzpicture}
}
\expandafter\newcommand\csname graph32\endcsname{%
\begin{tikzpicture}[scale=0.5,baseline=(current bounding box.center)]
\node[circle,draw,fill=yellow!70] (1) at (-0.0,1.0) {};
\node[circle, draw] (2) at (2.0,0.0) {};
\node[circle, draw] (3) at (1.0,0.0) {};
\node[circle, draw] (4) at (-0.0,0.0) {};
\graph {(3) -> (2); (1) -> (2); (4) -> (3); };
\end{tikzpicture}
}
\expandafter\newcommand\csname graph33\endcsname{%
\begin{tikzpicture}[scale=0.5,baseline=(current bounding box.center)]
\node[circle,draw,fill=yellow!70] (1) at (0.0,-1.0) {};
\node[circle, draw] (2) at (0.0,0.0) {};
\node[circle, draw] (3) at (1.0,-1.0) {};
\node[circle, draw] (4) at (-1.0,-1.0) {};
\graph {(2) -> (4); (2) -> (1); (2) -> (3); };
\end{tikzpicture}
}
\expandafter\newcommand\csname graph34\endcsname{%
\begin{tikzpicture}[scale=0.5,baseline=(current bounding box.center)]
\node[circle,draw,fill=yellow!70] (1) at (2.0,0.0) {};
\node[circle, draw] (2) at (-0.0,0.0) {};
\node[circle, draw] (3) at (1.0,0.5) {};
\node[circle, draw] (4) at (1.0,-0.5) {};
\graph {(2) -> (4); (3) -> (1); (2) -> (3); };
\end{tikzpicture}
}
\expandafter\newcommand\csname graph35\endcsname{%
\begin{tikzpicture}[scale=0.5,baseline=(current bounding box.center)]
\node[circle,draw,fill=yellow!70] (1) at (1.0,0.5) {};
\node[circle, draw] (2) at (-0.0,0.0) {};
\node[circle, draw] (3) at (1.0,-0.5) {};
\node[circle, draw] (4) at (2.0,0.0) {};
\graph {(2) -> (1); (3) -> (4); (2) -> (3); };
\end{tikzpicture}
}
\expandafter\newcommand\csname graph36\endcsname{%
\begin{tikzpicture}[scale=0.5,baseline=(current bounding box.center)]
\node[circle,draw,fill=yellow!70] (1) at (0.5,-1.0) {};
\node[circle, draw] (2) at (-0.5,0.0) {};
\node[circle, draw] (3) at (-0.5,-1.0) {};
\node[circle, draw] (4) at (0.5,0.0) {};
\graph {(4) -> (1); (2) -> (1); (2) -> (3); };
\end{tikzpicture}
}
\expandafter\newcommand\csname graph37\endcsname{%
\begin{tikzpicture}[scale=0.5,baseline=(current bounding box.center)]
\node[circle,draw,fill=yellow!70] (1) at (2.0,0.5) {};
\node[circle, draw] (2) at (1.0,0.0) {};
\node[circle, draw] (3) at (2.0,-0.5) {};
\node[circle, draw] (4) at (-0.0,0.0) {};
\graph {(4) -> (2); (2) -> (1); (2) -> (3); };
\end{tikzpicture}
}
\expandafter\newcommand\csname graph38\endcsname{%
\begin{tikzpicture}[scale=0.5,baseline=(current bounding box.center)]
\node[circle,draw,fill=yellow!70] (1) at (0.5,-1.0) {};
\node[circle, draw] (2) at (-0.5,0.0) {};
\node[circle, draw] (3) at (-0.5,-1.0) {};
\node[circle, draw] (4) at (0.5,0.0) {};
\graph {(4) -> (3); (2) -> (1); (2) -> (3); };
\end{tikzpicture}
}
\expandafter\newcommand\csname graph39\endcsname{%
\begin{tikzpicture}[scale=0.5,baseline=(current bounding box.center)]
\node[circle,draw,fill=yellow!70] (1) at (2.0,0.0) {};
\node[circle, draw] (2) at (-0.0,0.0) {};
\node[circle, draw] (3) at (1.0,0.0) {};
\node[circle, draw] (4) at (-0.0,1.0) {};
\graph {(3) -> (1); (4) -> (1); (2) -> (3); };
\end{tikzpicture}
}
\expandafter\newcommand\csname graph40\endcsname{%
\begin{tikzpicture}[scale=0.5,baseline=(current bounding box.center)]
\node[circle,draw,fill=yellow!70] (1) at (3.0,0.0) {};
\node[circle, draw] (2) at (1.0,0.0) {};
\node[circle, draw] (3) at (2.0,0.0) {};
\node[circle, draw] (4) at (-0.0,0.0) {};
\graph {(3) -> (1); (4) -> (2); (2) -> (3); };
\end{tikzpicture}
}
\expandafter\newcommand\csname graph41\endcsname{%
\begin{tikzpicture}[scale=0.5,baseline=(current bounding box.center)]
\node[circle,draw,fill=yellow!70] (1) at (2.0,0.0) {};
\node[circle, draw] (2) at (-0.0,-0.5) {};
\node[circle, draw] (3) at (1.0,0.0) {};
\node[circle, draw] (4) at (-0.0,0.5) {};
\graph {(3) -> (1); (4) -> (3); (2) -> (3); };
\end{tikzpicture}
}
\expandafter\newcommand\csname graph42\endcsname{%
\begin{tikzpicture}[scale=0.5,baseline=(current bounding box.center)]
\node[circle,draw,fill=yellow!70] (1) at (0.0,-1.0) {};
\node[circle, draw] (2) at (-1.0,0.0) {};
\node[circle, draw] (3) at (0.0,0.0) {};
\node[circle, draw] (4) at (1.0,0.0) {};
\graph {(3) -> (1); (4) -> (1); (2) -> (1); };
\end{tikzpicture}
}
\expandafter\newcommand\csname graph43\endcsname{%
\begin{tikzpicture}[scale=0.5,baseline=(current bounding box.center)]
\node[circle,draw,fill=yellow!70] (1) at (-0.0,0.0) {};
\node[circle, draw] (2) at (1.0,-0.5) {};
\node[circle, draw] (3) at (1.0,0.5) {};
\node[circle, draw] (4) at (2.0,0.0) {};
\graph {(2) -> (4); (1) -> (2); (1) -> (3); (3) -> (4); };
\end{tikzpicture}
}
\expandafter\newcommand\csname graph44\endcsname{%
\begin{tikzpicture}[scale=0.5,baseline=(current bounding box.center)]
\node[circle,draw,fill=yellow!70] (1) at (-0.5,0.0) {};
\node[circle, draw] (2) at (-0.5,-1.0) {};
\node[circle, draw] (3) at (0.5,-1.0) {};
\node[circle, draw] (4) at (0.5,0.0) {};
\graph {(1) -> (2); (4) -> (2); (1) -> (3); (4) -> (3); };
\end{tikzpicture}
}
\expandafter\newcommand\csname graph45\endcsname{%
\begin{tikzpicture}[scale=0.5,baseline=(current bounding box.center)]
\node[circle,draw,fill=yellow!70] (1) at (1.0,0.5) {};
\node[circle, draw] (2) at (2.0,0.0) {};
\node[circle, draw] (3) at (-0.0,0.0) {};
\node[circle, draw] (4) at (1.0,-0.5) {};
\graph {(1) -> (2); (3) -> (1); (4) -> (2); (3) -> (4); };
\end{tikzpicture}
}
\expandafter\newcommand\csname graph46\endcsname{%
\begin{tikzpicture}[scale=0.5,baseline=(current bounding box.center)]
\node[circle,draw,fill=yellow!70] (1) at (2.0,0.0) {};
\node[circle, draw] (2) at (-0.0,0.0) {};
\node[circle, draw] (3) at (1.0,0.5) {};
\node[circle, draw] (4) at (1.0,-0.5) {};
\graph {(2) -> (4); (3) -> (1); (4) -> (1); (2) -> (3); };
\end{tikzpicture}
}
\expandafter\newcommand\csname graph47\endcsname{%
\begin{tikzpicture}[scale=0.5,baseline=(current bounding box.center)]
\node[circle,draw,fill=yellow!70] (1) at (0.5,-1.0) {};
\node[circle, draw] (2) at (-0.5,0.0) {};
\node[circle, draw] (3) at (-0.5,-1.0) {};
\node[circle, draw] (4) at (0.5,0.0) {};
\graph {(4) -> (1); (4) -> (3); (2) -> (1); (2) -> (3); };
\end{tikzpicture}
}


\section{Introduction}
\label{sec:intro}
This article aims to produce a standardized framework for describing the relationships between multiple interdependent Mathematical Programs (a.k.a. optimization problems). Mathematical Programs can be viewed as formalism of a decision-making problem. A Mathematical Program Network is therefore a network of decision problems. In the study of multiple interacting decision-makers, the notion of equilibrium is central, i.e. the sets of decisions which are mutually acceptable for all parties. 

There is a robust field of work dedicated to studying problems involving multiple decisions, dating back at least to the early 19th century when Antoine Cournot was developing his theory of competition \cite{cournot1897researches}. Problems considering hierarchical relationships between decision-makers have been considered at least since Heinrich Stackelberg studied them in the early 1930s \cite{von2010market}. Throughout the mid to late 20th century, many researchers continued to build on these early foundations. The field has proved existence and uniqueness theorems for equilibria to various collections of decision problems, explored applications related to them, and derived algorithms to solve for such solutions. 

Nevertheless, the vast majority of existing research has been on relatively specific combinations of Mathematical Programs, focusing on variants of bilevel programs, static (Nash) games, single-leader/multi-follower problems, and equilibrium problems with equilibrium constraints or multi-leader/multi-follower problems. Comparatively, the amount of work on more elaborate networks of programs like multilevel optimization or multilevel games has been far less (``multi'' here meaning more than two). Presumably this is due to the understood complexity of multilevel problems, which as one might expect, increases in the number of levels. 

There have been some notable works which have nonetheless attempted to study and develop solution approaches for deeply hierarchical decision problems. Methods for solving trilevel linear optimization problems were developed in \cite{Bard:1984:ILT,Wen:1986:HAS}. Others studied the computational complexity of multilevel problems, proving that $(L+1)$-level linear problems are at least as hard as level $L$ of the polynomial time hierarchy \cite{Jeroslow:1985:PHS,Blair:1992:CCM}. Conditions required for the boundedness of multilevel linear problems were found in \cite{Benson:1989:SPL}. Necessary conditions for solutions to multilevel linear problems were described in \cite{ruan2004optimality}, in which interesting properties of the feasible set of solutions to these problems were also described. A thorough list of references related to multilevel optimization was collected in \cite{vicente1994bilevel}. Despite some of this analysis, none of these early approaches developed a general method for solving linear multilevel ($L>3$) optimization problems.

Recently, others have extended these early results to more general formulations. The work in \cite{faisca2009multi} attempted to develop a general method for solving multilevel quadratic problems, although the methods derived produce incorrect results in non-trivial settings. An evolutionary approach for general, unconstrained multilevel problems was developed in \cite{tilahun2012new}, and a related gradient-based approach proposed in \cite{sato2021gradient}. Both of these works take an empirical approach, and show that their approaches perform adequately in practice on some tasks, although they are relatively inefficient. Most recently, \cite{shafiei2021trilevel} presented convergence results for a method solving trilevel and multilevel problems via monotone operator theory. But like previous techniques, it requires that each level is an unconstrained decision problem, and assumes a specific structure on the objective functions of the decisions. In the realm of multilevel equilibrium problems, an approach for solving for Feedback Nash equilibria in repeated games was established in \cite{laine2021computation}, which was altered to consider Feedback Stackelberg equilibria in \cite{li2024computation}. However, both methods assume strict complementarity holds at every stage of the game, which is a strong assumption that can't be established prior to solving. 

Viewing the existing literature as a whole, it is clear there still exists a large gap in methods capable of finding equilibria for general networks of decision problem. Existing techniques for multilevel problems ($L>2$) assume either a single decision at each level, or assume a repeated game structure, strongly linking the decision nodes across levels of the network. More importantly, existing methods either assume that decision nodes do not involve constraints, or they do not properly account for the influence of said constraints on other decision problems in the network. Yet even so, it is not the lack of solution methods that is the biggest deficiency in the literature (at least not primarily)---rather, it is the absence of a formalism unifying the various problems involving multiple mathematical programs. Game Theory, i.e.  ``the study of mathematical models of conflict and cooperation between rational decision-makers''\cite{myerson2013game}, of course provides formalisms to describe relationships between rational decision-makers. A Mathematical Program Network, as defined here, could be thought of as a perfect-information extensive form game played over continuous strategy spaces, where the notion of equilibrium for these networks corresponds to a subgame perfect equilibrium \cite{selten1962a,selten1965b}. However, the literature is sparse with regards to these continuous games, especially those which involve non-trivial constraints on the strategy spaces of each player, or with non-isolated equilibria of the considered subgames. Whats more, except for certain cases like studies of duopolies \cite{daughety1990beneficial} or dynamic games \cite{bacsar1998dynamic}, existing Game Theory frameworks don't facilitate easy decoupling of the players in the game and the information pattern of the game. Considering nodes as players and edges as the underpinning of the information pattern, this distinction is made explicit in MPNs. 

The objective of this work is to formalize a framework for thinking about networks of Mathematical Programs. This framework should generalize the standard formulations such as Generalized Nash Equilibrium Problems \cite{debreu1952social,arrow1954existence,facchinei2007generalized}, Bilevel optimization problems \cite{bracken1973mathematical,vicente1994bilevel}, Mathematical Programs with Equilibrium Constraints \cite{outrata1994optimization}, Multi-Leader-Multi-Follower games \cite{sherali1984multiple}, and Feedback Nash Equilibrium problems \cite{laine2023computation}. This framework should provide a definition of equilibrium which aligns with the accepted notion of equilibrium for each of these special instances. Just as importantly as generalizing existing formulations, this framework should enable one to easily imagine alternate networks of decision problems which otherwise don't fit into the preexisting paradigms. Such a framework would enable researchers and practitioners to treat the relationship between decision-makers as a modeling choice just as the loss functions and feasible sets for each decision problem are modeling choices. In the study of duopolies, works such as \cite{bard1982explicit,daughety1990beneficial,huck2001stackelberg} considered the difference between modeling two firms in a Cournot vs Stackelberg competition. Ideally the proposed framework would enable a similar analysis and model exploration to occur for larger, more intricate collections of decision-makers.

Mathematical Program Networks, the topic of this article, aims to be such a framework. In the following sections, the case will be made that MPNs meaningfully fill the gap that is outlined above. In \cref{sec:notation}, preliminary notation that will be used throughout is outlined. In \cref{sec:formulation}, formal definitions of Mathematical Program Networks and their associated equilibrium points are provided. Key results are developed in \cref{sec:dev}, which will inspire the computational routines for finding equilibria presented in \cref{sec:comp}. Some examples of how MPNs can be used to model interesting problems are explored in \cref{sec:comp_examples}, including a large-scale empirical analysis on the impact network structure has on the cost of various equilibrium solutions. Concluding remarks are made in \cref{sec:conclusion}.


\section{Basic Notation}
\label{sec:notation}
For some integer $k$, let $[k]$ define the index set $\{1,\hdots, k\}$. For some vector $\x\in\mathbb{R}^n$, $x_j$ indicates the $j$th element of $\x$. For some index set $J \subset [n]$, let $[x_j]_{j\in J}$ denote the vector comprised of the elements of $\x$ selected by the indices in $J$. If $\x^j$, $j \in J$ is some collection of vectors, then $[\x^j]_{j\in J}$ denotes the concatenation of all vectors in that collection. For an index set $J$, $\mathcal{P}(J)$ is the set of all subsets of $J$ (power set). 

For some set $S \subset \R{n}$, the closure of $S$ is denoted $\bar{S}$, the complement of $S$ ($\R{n} \setminus S$) is denoted $S'$, the interior of $S$ ($\R{n} \setminus \closed{(\R{n} \setminus S)}$) is $\interior{S}$, and the boundary of $S$ ($\closed{S} \setminus \interior{S}$) is denoted $\partial S$. The dual cone of $S$ is denoted $S^* := \{\mathbf{v}\in\R{n} : \langle \mathbf{v},\mathbf{d}\rangle \geq 0 \ \forall \mathbf{d} \in S\}$, where $\langle \cdot, \cdot \rangle$ is the standard inner product on $\mathbb{R}^n$. The convex hull of a set of points $\mathbf{p}^j$, $j\in J$ is denoted $\mathrm{conv}(\{\mathbf{p}^j\}_{j\in J})$. The conic hull of a set of vectors $\mathbf{v}^j$, $j\in J$ is denoted $\mathrm{coni}(\{\mathbf{p}^j\}_{j\in J})$.

An open $\epsilon$-ball around some point $\x \in \R{n}$ is given as $\mathcal{N}_\epsilon(\x) := \{\mathbf{y} \in \R{n}: \|\mathbf{y}-\x\| < \epsilon \}$.






\section{Formulation}
\label{sec:formulation}
\begin{figure}[t]
    \centering
    \fbox{\includegraphics[angle=-90,origin=c,width=0.90\linewidth]{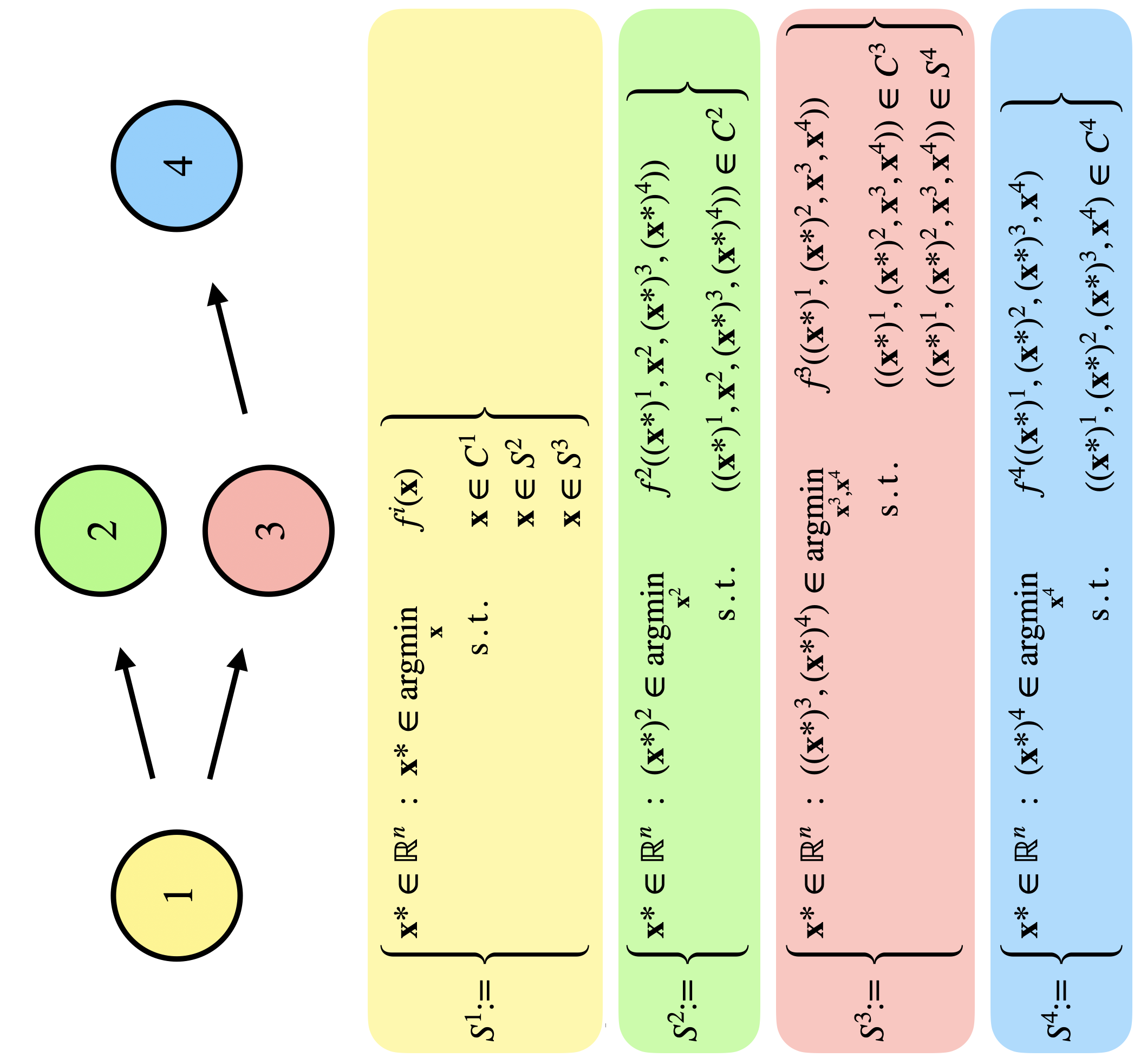}}
    \caption{A four-node Mathematical Program Network. Each node attempts to minimize a cost function $f^i$ while satisfying the constraints represented by $C^i$ as well as those which are imposed by the network structure. An equilibrium is any point which is an element of all four solution graphs $S^i$, which in this case is equivalent to being an element of $S^1$.}
    \label{fig:front}
\end{figure}

\begin{definition}[MPN] \label{def:MPN} A Mathematical Program Network (MPN) is defined by the tuple 
    \begin{equation*}
        (\{f^i, C^i, J^i\}_{i \in [N]}, E),
    \end{equation*}
    and represents a directed graph comprised of $N$ decision nodes (mathematical programs). The nodes jointly optimize a vector of decision variables $\mathbf{x} \in \mathbb{R}^n$. Each decision node $i \in [N]$ is represented by a cost function $f^i : \mathbb{R}^n \to \mathbb{R}$, a feasible set $C^i \subset \mathbb{R}^n$, and a decision index set ${J}^i \subset [n]$. The network structure of the MPN is defined in terms of the set of directed edges $E \subset [N] \times [N]$. The edge $(i,j) \in E$ iff node $j$ is a child of node $i$. 
\end{definition}

For ease of notation, some additional sets and vectors are defined. Let $n^i$ be the cardinality of $J^i$. Then the sub-vectors $\x^i$ of $\mathbf{x}$ are said to be the private decision variables for node $i$:
\begin{align*}
    \mathbf{x}^i &\in \mathbb{R}^{n^i} \ \ := [x_j]_{j\in{J}^i}.
\end{align*}

 The set of reachable transitions for an MPN is represented as $R \subset [N] \times [N]$. Specifically, $(i,j) \in R$ iff there exists a path from node $i$ to node $j$ by traversing edges in $E$. The set of reachable transitions is used to define the following useful sets:
 \begin{equation} \label{eq:Di}
\begin{aligned}
    D^i &:= \{i\} \cup \{j : (i,j) \in R \} \\
    D^{-i} &:= [N] \setminus D^i
\end{aligned}
\end{equation}

The vector $\x$ can be indexed according to these sets:
\begin{align*}
    \mathbf{x}^{D^i} &:= [\x^j]_{j\in D^i} \\
    \mathbf{x}^{D^{-i}} &:= [\x^j]_{j\notin D^i},
\end{align*}
The vector $\x^{D^i}$ is the concatenation of all private decision variables for node $i$ and its descendants. The vector $\x^{D^{-i}}$ are all other decision variables. These sub-vectors are often referenced in conjunction with the following shorthand. For any two vectors $\mathbf{x}$ and $\mathbf{x}^*$, the following identity holds:
\begin{equation*}
    \left(\mathbf{x}^{D^i}, (\mathbf{x}^*)^{D^{-i}}\right) \equiv \mathbf{y} \in \R{n},
\end{equation*}
where 
\begin{equation*}
    y_j := \begin{cases}
        x_j: & \exists k \in D^i, j \in J^k \\
        x^*_j: & \mathrm{else}.
    \end{cases}
\end{equation*}
This identity can be thought of as a way to combine $\x^{D^i}$ and $(\x^*)^{D^{-i}}$ into a single vector comparable to $\x$ or $\x^*$. Using all of the above-defined terms, the defining object of a MPN node, the solution graph, can be defined. The solution graph for some node $i \in [N]$ is the following:
\begin{equation}
    \label{eq:solution_graph}
        S^i := \left\{ \begin{aligned} \mathbf{x}^* \in \mathbb{R}^n \ : \ \left(\mathbf{x}^*\right)^{D^i} \in &\argmin_{\mathbf{x}^{D^i}}  && f^i\left(\mathbf{x}^{D^i}, \left(\mathbf{x}^*\right)^{D^{-i}}\right) \\
        & \ \ \ \ \text{s.t.} && \left(\mathbf{x}^{D^i}, \left(\mathbf{x}^*\right)^{D^{-i}}\right) \in C^i \\
        &&& \left(\mathbf{x}^{D^i}, \left(\mathbf{x}^*\right)^{D^{-i}}\right) \in S^j, \ (i,j) \in E \end{aligned} \right\}
\end{equation}

The $\argmin$ in \cref{eq:solution_graph} is taken in a local sense, so that $S^i$ defines the set of all $\mathbf{x}$ for which $\mathbf{x}^{D^i}$ are local minimizers to the optimization problem parameterized by $\mathbf{x}^{D^{-i}}$. The role that network architecture plays in each node's optimization problem is made clear from \cref{eq:solution_graph}. The private decision variables of child nodes are assumed by the parent node(s), with the requirement that the parents are constrained by the solution graphs of the children.

\begin{definition}[Equilibrium] \label{defn:equilibrium}
    A point $\mathbf{x}^* \in \mathbb{R}^n$ is an Equilibrium of a MPN iff $x^*$ is an element of each node's solution graph. Specifically, $\mathbf{x}^*$ in an equilibrium iff $\mathbf{x}^* \in S^*$ ,where
    \begin{equation} \label{eq:equilibrium}
         S^* := \bigcap_{i \in[N]} S^i.
    \end{equation}
\end{definition}

The solution graph of any mathematical program is a subset of its feasible set. Therefore, $S^i \subset S^j \ \forall (i,j) \in E$, therefore $S^i \subset S^j \ \forall (i,j) \in R$, and the following holds:

\begin{corollary}\label{cor:source_nodes}
    Let $I^{\mathrm{source}} \subset [N]$ be any index set such that $\forall j \in [N], \exists i \in I^{\mathrm{source}} : (i,j) \in R \ $. Then 
    \begin{equation}
        S^* = \bigcap_{i \in I^{\mathrm{source}}} S^i.
    \end{equation}
\end{corollary}

A \emph{cyclic network} is a network for which $(i,i) \in R$ for some $i\in[N]$. An \emph{acyclic network} is then a network which is not cyclic. For acyclic MPNs, the index sets in \cref{cor:source_nodes} correspond to the set of all source node indices, where a source node is a node without any incoming edges. Cyclic MPNs are considered to be degenerate for the following reason.

\begin{corollary}\label{cor:cyclic}
    Let $I^\mathrm{conn} \subset [N]$ be a set of fully-connected nodes in an MPN, meaning ($i,j) \in R, \ \forall i,j \in I^\mathrm{conn}$. Let $I^\mathrm{downstream} := \{j \notin I^\mathrm{conn}: \exists i \in I^\mathrm{conn} \ \mathrm{s.t.} \ (i,j) \in R\}$ be the set of all nodes which are descendants of the nodes in $I^\mathrm{conn}$ (but not themselves in $I^\mathrm{conn}$). Then any singleton set of the form
    \begin{equation} \label{eq:cor_cyclic}
        S^\mathrm{conn} = \{ \mathbf{x} \}, \ \mathbf{x} \in \bigg(\bigcap_{j \in I^\mathrm{conn}} C^j \bigg) \cap \bigg(\bigcap_{j \in I^\mathrm{downstream}} S^j\bigg)
    \end{equation} is a valid solution graph for all $i \in I^\mathrm{conn}$. 
\end{corollary}

    


The result of \cref{cor:cyclic} is essentially that solution graphs (and therefore equilibrium points) are not meaningfully defined on cyclic network configurations. The cyclic nature of the network translates to a cyclic definition in what it means to be a solution graph, and the resulting circular definition allows for degenerate graphs which dilute the function of the decision nodes. Cyclic MPNs and the resulting equilibrium points are related to the notion of Consistent Conjectures Equilibria \cite{bresnahan1981duopoly}. \Cref{cor:cyclic} mirrors some of the criticism in the literature pertaining to Conjectural Variations and some of the logical difficulties they impose when trying to think of them in a rationality or mathematical programming framework \cite{makowski1987arerational,lindh1992inconsistency}. For these reasons, only acyclic networks are considered in the remainder of this work.

\subsection{Quadratic Program Networks}
Particular emphasis in this article is given to a special case of MPNs which possess desirable properties from an analysis and computation perspective. These special MPNs are termed Quadratic Program Networks, defined as the following:

\begin{definition}[QPNs] \label{defn:qpn}
    A Quadratic Program Network (QPN) is a MPN in which the cost function $f^i$ for each decision node is a quadratic function that is convex with respect to the variables $\x^{D^i}$, and the feasible region $C^i$ is a convex (not-necessarily closed) polyhedral region. In other words, each node, when considered in isolation, is a quadratic program.
\end{definition}

In the above definition, a not-necessarily closed polyhedral region is defined as a finite intersection of not-necessarily closed halfspaces, i.e. sets of the form
\begin{equation} \label{eq:slice}
    H \subset \mathbb{R}^n := \{ \x: \langle \mathbf{a}, \x \rangle \gneq b \},
\end{equation} 
where, $\mathbf{a}\in\mathbb{R}^n$, $b \in \mathbb{R}$, and $\gneq \in \{>,  \geq \}$ is used to indicate that the halfspace may or may not be a closed set. Therefore a not-necessarily convex polyhedral region $P$ can be expressed as
\begin{equation} \label{eq:poly}
    P \subset \mathbb{R}^n := \bigcap_{i\in[m]} H^i.
\end{equation}

It will be seen in later sections that QPNs are an important subclass of MPN, since the solution graphs of their constituent nodes are always unions of convex polyhedral regions. This property lends itself to computation, whereas the the solution graphs of general MPNs can quickly become overwhelmingly complex. Even though the QPN framework imposes considerable restrictions on the type of nodes allowed, they retain a remarkable degree of modeling flexibility. This is demonstrated in the following section. 

\section{Development}
\label{sec:dev}
In this section, the focus turns towards deriving results which will lead to algorithms for computing equilibrium points for MPNs. The development is centered around the representation of the solution graphs of each node, which are central to the definition of equilibrium (\cref{defn:equilibrium}). 

Specifically, it is shown that each solution graph of a MPN node can be represented as a union of regions derived from the solution graphs of its children. Starting with the childless nodes in an (acyclic) MPN, conditions of optimality lead to a representation of solution graph for those nodes. These representations can then be used to define optimality conditions for those nodes' parents, leading in turn to representations of their solution graphs. An inductive argument leads to representations of the solution graph for all nodes in the network, and therefore a representation for the set of equilibrium points. 

In order to establish this procedure, several intermediate results are established. First, consider a generic parametric optimization problem and its associated solution graph, 
\begin{equation} \label{eq:basic_opt}
    S := \left\{ 
    \begin{aligned} (\x^*, \w) : \ \x^* \in &\argmin_{\x}  && f(\x, \w) \\
    & \ \ \ \ \mathrm{s.t.} && (\x,\w) \in C \\
    \end{aligned} 
    \right\}
\end{equation}
where $\w \in \R{m}$ is some parameter vector, $f: \R{n}\times\R{m} \to \R{}$ is a continuous cost function and $C\subset \R{n}
\times\R{m}$ defines a feasible region. For a given $\w$, let $C_\w$ denote the slice of $C$ at $\w$, i.e. $C_\w := \{\x \in \R{n} : (\x, \w)  \in C\}$.

\begin{definition}
    For a given $\w$, a point $\x^*$ with $(\x^*, \w) \in C$ is a local optimum for (\ref{eq:basic_opt}), i.e. $(\x^*, \w) \in S$, if there exists some $\epsilon > 0$ such that 
    \begin{equation}
        f(\x^*, \w) \le f(\x, \w) \ \forall \x \in \mathcal{N}_\epsilon(\x^*) \cap C_\w.
    \end{equation}
\end{definition}

\begin{lemma}\label{lemma:closure}
    Define $\hat{S}$ as the following: 
    \begin{equation}\label{eq:basic_opt_closed}
    \hat{S} := \left\{ 
    \begin{aligned} (\x^*, \w) : \ \x^* \in &\argmin_{\x}  && f(\x, \w) \\
    & \ \ \ \ \mathrm{s.t.} && (\x,\w) \in \closed{C} \\
    \end{aligned} 
    \right\}.
    \end{equation}
    Then 
    \begin{equation}
        S = \hat{S} \cap C.
    \end{equation}
\end{lemma}
Recall $\closed{C}$ is the closure of $C$.
\begin{proof}
    Let $\closed{C}_\w$ be defined analogously to $C_\w$, i.e. the slice of $\closed{C}$ at $\w$. If $(\x^*,\w) \in \hat{S} \cap C$, then 
    \begin{equation}
        f(\x^*,\w) \le f(\x, \w) \ \forall \x \in \mathcal{N}_\epsilon(\x^*) \cap \closed{C}_\w,
    \end{equation}
    and since $C_\w \subset \closed{C}_\w$, it is clear that $\hat{S} \cap C \subset S$. For the reverse implication, if $(\x^*, \w) \in S$, then $(\x^*,\w) \in C$, so it must be shown that $S \subset \hat{S}$. This is proved via contradiction.

    For a given $\w$, suppose there exists some $\x^*$ such that $(\x^*, \w) \in S$, $(\x^*, \w) \notin \hat{S}$. Assume without loss of generality that $f(\x^*, \w) = 0$. This implies the existence of some $\epsilon > 0$ and $\hat{\x}$ such that $f(\hat{\x},\w) < 0$, with $\hat{\x} \in \N_{\epsilon}(\x^*) \cap (\closed{C}_\w \setminus C_\w)$. Note that $(\closed{C}_\w \setminus C_\w) \subset \partial C_\w$. Let $\delta = |f(\hat{\x},\w)|$. By the continuity of $f$, there exists an open neighborhood of $\hat{\x}$ such that $|f(\x, \w) - f(\hat{\x}, \w)| < \delta$ for all points $\x$ in this neighborhood, but since $\hat{\x} \in \partial C_\w$, every open neighborhood of $\hat{\x}$ contains at least one point in $C_\w$. Since $\hat{\x} \in \N_\epsilon(\x^*)$, an open set, all sufficiently local neighborhoods of $\hat{\x}$ are also subsets of $\N_\epsilon(\x^*)$. This implies that there exists at least one point $\x \in \N_\epsilon(\x^*) \cap C_\w$ with $|f(\x,\w) - f(\hat{\x},\w)| < \delta$, i.e. $f(\x,\w) < 0$, which contradicts $\x^*$ being a local optimum for (\ref{eq:basic_opt}). Therefore, the assumption must be incorrect, and $S \subset \hat{S}$ as desired.  \qed
\end{proof}

\begin{lemma}\label{lemma:local}
    For arbitrary $\tilde{\x} \in \R{n}$,  $\tilde{\w} \in \R{m}$, and $\epsilon > 0$, let $\tilde{S}_\epsilon(\hat{\x}, \hat{\w})$ be defined as the following: 
    \begin{equation}\label{eq:basic_opt_local}
    \tilde{S}_\epsilon(\tilde{\x}, \tilde{\w}) := \left\{ 
    \begin{aligned} (\x^*, \w) : \ \x^* \in &\argmin_{\x}  && f(\x, \w) \\
    & \ \ \ \ \mathrm{s.t.} && (\x,\w) \in C \cap \N_\epsilon(\tilde{\x}, \tilde{\w}) \\
    \end{aligned} 
    \right\}.
    \end{equation}
    Then 
    \begin{equation}
        S \cap \mathcal{N}_\epsilon(\tilde{\x},\tilde{\w}) = \tilde{S}_\epsilon(\tilde{\x}, \tilde{\w})
    \end{equation}
\end{lemma}
\begin{proof}
    Consider some $(\x^*, \w) \in \tilde{S}_\epsilon(\tilde{\x}, \tilde{\w})$. Then by definition, $(\x^*, \w) \in C \cap \mathcal{N}_\epsilon(\tilde{\x},\tilde{\w})$, and there exists some $\epsilon' > 0$ such that
    \begin{align}
        f(\x^*,\w) \le f(\x, \w) \ \forall \x \in \mathcal{N}_{\epsilon'}(\x^*) \cap C_\w  \cap (\N_{\epsilon}(\tilde{\x},\tilde{\w}))_{\w},
    \end{align}
    since $(C \cap \N_{\epsilon}(\tilde{\x},\tilde{\w}))_\w =  C_\w  \cap (\N_{\epsilon}(\tilde{\x},\tilde{\w}))_{\w}$. Because $(\tilde{\x},\tilde{\w}))_{\w}$ is an open set containing $\x^*$, there exists a sufficiently small choice of $\epsilon'$ such that $\mathcal{N}_{\epsilon'}(\x^*) \subset (\N_\epsilon(\tilde{\x},\tilde{\w}))_{\w}$. But this implies 
    \begin{align}
        f(\x^*,\w) \le f(\x, \w) \ \forall \x \in \mathcal{N}_{\epsilon'}(\x^*) \cap C_\w,
    \end{align}
    which implies $(\x^*, \w) \in S \cap \mathcal{N}_\epsilon(\tilde{\x},\tilde{\w})$, and therefore $\tilde{S}_\epsilon(\tilde{\x}, \tilde{\w}) \subset S \cap \mathcal{N}_\epsilon(\tilde{\x},\tilde{\w})$. 

    Now, for the reverse implication, consider some $(\x^*, \w) \in S\cap \mathcal{N}_\epsilon(\tilde{\x},\tilde{\w})$. Then $(\x^*,\w) \in C$, $(\x^*, \w) \in \N_\epsilon(\tilde{\x}, \tilde{\w})$, and there exists some $\epsilon'$ such that
    \begin{equation}
        f(\x^*, \w) \le f(\x, \w) \ \forall \x \in C_\w \cap \N_{\epsilon'}(\x^*).
    \end{equation}
    Since $C_\w \subset C_\w \cap (\N_{\epsilon}(\tilde{\x},\tilde{\w}))_\w = (C \cap \N_{\epsilon}(\tilde{\x},\tilde{\w}))_\w$, this also implies $(\x^*,\w) \in \tilde{S}_\epsilon(\tilde{\x},\tilde{\w})$, and $S \cap \mathcal{N}_\epsilon(\tilde{\x},\tilde{\w}) \subset \tilde{S}_\epsilon(\tilde{\x},\tilde{\w})$. Therefore, 
    \begin{equation}
        S \cap \mathcal{N}_\epsilon(\tilde{\x},\tilde{\w}) \subset \tilde{S}_\epsilon(\tilde{\x},\tilde{\w}) \subset S \cap \mathcal{N}_\epsilon(\tilde{\x},\tilde{\w}).
    \end{equation} \qed
\end{proof}

Now, consider a region defined to be the finite union of a number of other regions: 
\begin{equation} \label{eq:Udef}
    U := \bigcup_{j \in J} U^j,
\end{equation}
with $U^j \subset \R{n}\times\R{m}$, and $J \subset \mathbb{N}$. Using these sets, define the following programs and associated solution graphs, where $f$ is again a continuous function:
\begin{equation}\label{eq:MP}
    R := \left\{ 
    \begin{aligned} (\x^*, \w) : \ \x^* \in &\argmin_{\x}  && f(\x, \w) \\
    & \ \ \ \ \mathrm{s.t.} && (\x,\w) \in U \\
    \end{aligned} 
    \right\}
\end{equation}
and for each $j \in J$,
\begin{equation}\label{eq:MPj}
    R^j := \left\{ 
    \begin{aligned} (\x^*, \w) : \ \x^* \in &\argmin_{\x}  && f(\x, \w) \\
    & \ \ \ \ \mathrm{s.t.} && (\x,\w) \in \closed{U^j} \\
    \end{aligned} 
    \right\}.
\end{equation}
Let $U_\w$ and $\closed{U^j}_\w$ be defined analogously to $C_\w$. For some $\x \in \mathbb{R}^n$ and $\w \in \R{m}$, define 
\begin{equation}
    \gamma_{U}(\x, \w) := \{j \in J: (\x,\w) \in \closed{U^j} \}.
\end{equation}

\begin{lemma} \label{lem:gamma_subset}
    Given a set $U$ of the form (\ref{eq:Udef}), a sufficiently small $\epsilon$, and any $(\x, \w)$,
    \begin{equation}
    \gamma_U(\hat{\x},\hat{\w}) \subset \gamma_U(\x, \w), \ \ \forall (\hat{\x}, \hat{\w}) \in \N_\epsilon(\x, \w).
    \end{equation}
\end{lemma}
\begin{proof}
    Note that $(\x,\w) \in (\closed{U^j})' \ \forall j \notin \gamma_U(\x,\w)$, and all sets $(\closed{U^j})'$ are open. Therefore for sufficiently small $\epsilon$, 
    \begin{equation}
        (\hat{\x},\hat{\w}) \in (\closed{U^j})' \ \forall j \notin \gamma_U(\x,\w), \ \forall (\hat{\x},\hat{\w}) \in \N_\epsilon(\x, \w).
    \end{equation}
    This implies the result. \qed
\end{proof}

\begin{lemma} \label{lem:unions}
    Given the definitions of $R$ and $R^j$ in (\ref{eq:MP}) and (\ref{eq:MPj}), 
    \begin{equation}
        (\x^*,\w) \in R \iff \left( (\x^*,\w) \in U, \text{ and } (\x^*, \w) \in R^j \ \forall j \in \gamma_{U}(\x^*, \w) \right).
    \end{equation}
\end{lemma}
\begin{proof}
    Applying \cref{lemma:closure}, we have that $\x^*$ is a local opt for (\ref{eq:MP}) iff  $\x^* \in U_\w$ and $\exists \epsilon$: 
    \begin{equation} \label{eq:opt_union_closed}
    f(\x^*,\w) \le f(\x,\w) \ \forall \x \in \mathcal{N}_\epsilon(\x^*) \cap \closed{U}_\w.
    \end{equation} 
    Note that $\closed{U}_\w = \bigcup_{j\in J} \closed{U}^j_\w$. Therefore, by the definition of $\gamma_U$, $\exists \epsilon > 0$ such that
    \begin{equation} \label{eq:opt_j_closed}
        (\x^*, \w) \in \closed{U^j}, \ f(\x^*,\w) \le f(\x,\w) \ \forall \x \in \mathcal{N}_\epsilon(\x^*) \cap \closed{U^j}_\w, \ \forall j \in \gamma_U(\x^*, \w).
    \end{equation}
    This proves the $\Rightarrow$ direction. To prove the $\Leftarrow$ direction, it is already stated that $(\x^*, \w) \in U$, so all that is left to prove is that the condition (\ref{eq:opt_j_closed}) implies (\ref{eq:opt_union_closed}). To see this, note that for sufficiently small $\epsilon$, 
    \begin{equation} \label{eq:local_closed_U}
        \closed{U} \cap \N_\epsilon(\x^*, \w) =\bigcup_{j \in J} \closed{U^j} \cap \N_\epsilon(\x^*, \w) = \bigcup_{j \in \gamma_U(\x^*, \w)} \closed{U^j} \cap \N_\epsilon(\x^*, \w),
    \end{equation}
    where the last equality is implied in one direction by the fact that $\gamma_U(\x^*,\w) \subset J$, and in the other direction by \cref{lem:gamma_subset}. This further implies that 
    \begin{equation}
        \closed{U}_\w \cap \N_\epsilon(\x^*) = \bigcup_{j \in \gamma_U(\x^*, \w)} \closed{U^j}_\w \cap \N_\epsilon(\x^*),
    \end{equation} which when taken with (\ref{eq:opt_j_closed}), results in $(\ref{eq:opt_union_closed}$ as desired. \qed
\end{proof}

\begin{corollary}\label{cor:sol_union}
      Given the definitions of $R$ and $R^j$ in (\ref{eq:MP}) and (\ref{eq:MPj}),
      \begin{equation}
          R = \bigcup_{\substack{J_1 \in \mathcal{P}(J)\\J_2 = J \setminus J_1}} \left( U \bigcap_{j_1 \in J_1} R^{j_1} \bigcap_{j_2 \in J_2} \left(\closed{U^{j_2}}\right)' \right).
      \end{equation}
\end{corollary}
\begin{proof}
    For some $J_1 \subset J$, $J_2 = J \setminus J_1$, and a set of the form
    \begin{equation}
        \Gamma_U(J_1,J_2) := \bigcap_{j_1 \in J_1} \closed{U^{j_1}} \bigcap_{j_2 \in J_2} \left(\closed{U^{j_2}}\right)',
    \end{equation}
    \cref{lem:unions} states that
    \begin{equation}
        R \cap \Gamma_U(J_1, J_2) = U \bigcap_{j_1 \in J_1} R^{j_1} \bigcap_{j_2 \in J_2} \left(\closed{U^{j_2}}\right)'.
    \end{equation}
    Considering all possible sets $J_1 \subset J$ and $J_2 = J \setminus J_1$ gives the result. \qed 
\end{proof}
    
\begin{corollary}\label{cor:sol_union_local}
      Given the definitions of $R$ and $R^j$ in (\ref{eq:MP}) and (\ref{eq:MPj}), there exists $\epsilon > 0$ such that
      \begin{equation} \label{eq:sol_union_local}
          R \cap \N_\epsilon(\x, \w) = \bigcup_{\substack{J_1 \in \mathcal{P}(\gamma_U(\x,\w))\\J_2 = \gamma_U(\x,\w) \setminus J_1}} \left( U \bigcap_{j_1 \in J_1} R^{j_1} \bigcap_{j_2 \in J_2} \left(\closed{U^{j_2}}\right)' \bigcap \N_\epsilon(\x,\w) \right).
      \end{equation}
\end{corollary}
\begin{proof}
    This follows immediately from \cref{cor:sol_union} and noting that \cref{lem:gamma_subset} implies that for all $j \notin \gamma_U(\x,\w)$, $\N_\epsilon(\x, \w) \bigcap \closed{U^{j}} = \varnothing$, and therefore since $R^j \subset \closed{U^j}$, $\N_\epsilon(\x, \w) \bigcap R^j = \varnothing$. Furthermore, $\N_\epsilon(\x,\w) \subset (\closed{U^{j}})'$. Together, it's made clear that the entire set $J$ need not be considered, but rather only the local subset $\gamma_U(\x,\w)$. \qed
\end{proof}

\Cref{cor:sol_union_local} states that a local representation of the solution graph (\ref{eq:MP}) can be constructed using local representations of the solution graphs (\ref{eq:MPj}). This is significant, since the solution graphs $R^j$ can often be easily represented in cases of practical interest. This fact serves as the foundation for testing whether points are equilibrium solutions for MPNs as well as computing solutions, as described below.

\subsection{Quadratic Program Networks}

QP Networks have some desirable properties, which when taken together with the results in the preceding section, can be leveraged to derive algorithms for identifying equilibrium points. These properties are summarized in the following claims:

\begin{lemma}
    \label{lemma:QP_sol}
    Consider an optimization problem of the form (\ref{eq:basic_opt}), where $f$ is a quadratic function and convex with respect to $\x$, and $C$ is a closed, non-empty polyhedron in $\R{n} \times \R{m}$, i.e. 
    \begin{equation}
        C = \bigcap_{i \in [l]} H^i,
    \end{equation}
    where, for each $i \in [l]$, and some $\mathbf{a}^i \in \R{n}$, and $\mathbf{b}^i \in \R{m}$, $c^i \in \R{}$,
    \begin{equation}
        H^i = \left\{ (\x,\w) \in \R{n} \times \R{m} : \langle \x, \mathbf{a}^i \rangle + \langle \w, \mathbf{b}^i \rangle \ge c^i \right\}.
    \end{equation}
    Then the solution graph (\ref{eq:basic_opt}) is a union of polyhedral regions, and is given by 
    \begin{equation} \label{eq:sol_graph_QP}
        S = \bigcup_{L \in \mathcal{P}([l])} 
        \left\{ (\x,\w) \in C: \begin{aligned}
            &(\x,\w) \in \cap_{i \in L} \partial{H^{i}} \\
            &\nabla_x f(\x,\w) \in \mathrm{coni}( \{ \mathbf{a}^i \}_{i \in L})
        \end{aligned}
         \right\}.
    \end{equation}
\end{lemma}
\begin{proof}
    Because the problem (\ref{eq:basic_opt}) is a convex optimization problem with linear constraints, the set of solutions are given by the necessary and sufficient conditions of optimality,
    \begin{equation} \label{eq:qp_ocs}
        S = \left\{ (\x,\w) \in C : \nabla_\x f(\x,\w) \in (C_\w - \x)^*   
         \right\}.
    \end{equation}
    Note that $C_\w - \x$ is the tangent cone to $C_\w$ at $\x$. The dual cone of an intersection of sets is given by the Minkowski sum of the dual of each set being intersected. Therefore 
    \begin{equation} \label{eq:dual_of_int}
        (C_\w - \x)^* = \sum_{i=1}^l (H^i_\w - \x)^*,
    \end{equation}
    where each expression in the sum is given as
    \begin{equation}
        (H^i_\w - \x)^* = \begin{cases}
            \{\mathbf{0}\}, & \x \in \interior{H_\w^i} \\
            \{t \cdot \mathbf{a}^i,\ \forall t\ge 0\} , &\x \in \partial H_\w^i.
        \end{cases}
    \end{equation}
    Hence, the only terms contributing to the sum in (\ref{eq:dual_of_int} are those for which $\x \in \interior{H_\w^i}$. Therefore
    \begin{equation}
        (C_\w - \x)^* = \left\{ \left(\sum_{i: \x \in \interior{H_\w^i}} t^i \cdot \mathbf{a}^i\right), t^i \ge 0 \right\} = \mathrm{coni}( \{ \mathbf{a}^i \}_{i : \x \in \interior{H_\w^i}}).
    \end{equation}
    Using the fact that $\x \in \partial H_\w^i \iff (\x,\w) \in \partial H^i$, and considering all possible boundaries $\partial H^i$ that $(\x,\w)$ can lie within, the result follows. To see that each of the sets in the union (\ref{eq:sol_graph_QP}) are polyhedral, note that the dual cone can be expressed in halfspace representation, i.e. there exist vectors $\{\mathbf{y}^j\}_{j \in J_L}$ such that
    \begin{equation} \label{eq:dual_cone_duality}
        \mathbf{g} \in \mathrm{coni}( \{ \mathbf{a}^i \}_{i \in L}) \iff \langle \mathbf{g}, \mathbf{y}^j \rangle \ge 0, \forall j \in J_L.
    \end{equation}
    Since $f$ is quadratic, $\nabla_\x f(\x,\w)$ is an affine expression, i.e. can be expressed as $\mathbf{Q}\x + \mathbf{R}\w + \mathbf{q}$, for some matrices $\mathbf{Q}\in\R{n\times n}, \mathbf{R}\in\R{n\times m }$, and vector $\mathbf{q} \in \R{n}$.
    Therefore, (\ref{eq:sol_graph_QP}) can equivalently be expressed as 
    \begin{equation} \label{eq:sol_graph_QP_hrep}
        S = \bigcup_{L \in \mathcal{P}([l])} 
        \left\{ (\x,\w) \in C: \begin{aligned}
            &\langle \x, \mathbf{a}^i \rangle + \langle \w, \mathbf{b}^i \rangle = c^i, \ &&\forall i \in L\\
            &\langle \x, \mathbf{Q}^\intercal \mathbf{y}^j \rangle + \langle \w, \mathbf{R}^\intercal \mathbf{y}^j \rangle \ge -\langle \mathbf{q}, \mathbf{y}^j \rangle, \ &&\forall j \in J_L 
        \end{aligned}
         \right\},
    \end{equation}
    which is clearly a union of polyhedral regions. \qed
\end{proof}

\begin{lemma}\label{eq:poly_closure}
    Consider some not-necessarily closed polyhedral region $P$ of the form (\ref{eq:poly}), i.e. the intersection of not-necessarily closed halfspaces $H^i$, $i \in [m]$. If $\interior{P} \ne \emptyset$, then 
    \begin{equation}
        \closed{P} = \bigcap_{i \in [m]} \closed{H^i}.
    \end{equation}
\end{lemma}
\begin{proof}
 As in \cite{rockafellar2015convex}, theorem 6.5.
\end{proof}

\begin{lemma} \label{lemma:QPN_sol}
    The solution graph of any node $i$ in an acyclic QPN can be expressed as a union of not-necessarily closed polyhedral regions, i.e.
    \begin{equation} 
        S^i = \bigcup_{j \in J_i} P^j,
    \end{equation}
    where $J_i$ is some index set, and each set $P^j$ is of the form (\ref{eq:poly}). 
\end{lemma}
\begin{proof}
    Consider first any node $j$ without children. As with all nodes in a QPN, the region $C^j$ is polyhedral and not-necessarily closed. Then \cref{lemma:closure,lemma:QP_sol} state that the solution graph must be given by the intersection of $C^j$ and a set of the form (\ref{eq:sol_graph_QP}), which results in a union of polyhedral regions, which each may not be closed if $C^j$ is not closed. 

    Now consider any node $i$ which does have children, and assume that the solution graphs of its children are unions of the above form. It will be seen that the solution graph of node $i$ is also of that form. First, note that the decision problem for any such node is of the form (\ref{eq:MP}), where the set $U$ is the intersection of $C^i$ and all solution graphs of its children. Hence $U$ is a union of not-necessarily closed polyhedral regions. \Cref{cor:sol_union} states that the solution graph is given by a union of sets formed by intersecting $U$ (union of not-necessarily closed polyhedra), sets of the form $R^{j_1}$ (unions of closed polyhedra), and sets of the form $(\closed{U^{j_2}})'$ (unions of open polyhedra). This intersection gives rise again to unions of not-necessarily closed polyhedra, implying the result. \qed
\end{proof}
\begin{corollary}
    \label{cor:QPN_equilibrium}
    The set of equilibria for a QPN is a union of not-necessarily closed polyhedral regions.
\end{corollary}
\begin{proof} 
    This follows imediately from \cref{lemma:QPN_sol} and \cref{defn:equilibrium}. \qed
\end{proof}

The preceding results imply computational routines for checking whether a point is an equilibrium for a QPN, and similarly, searching for an equilibrium. These schemes are described in the next section.

\section{QPN Equilibrium Computation}
\label{sec:comp}
The results of the preceding section established that the solution graph for each node in a QPN is a finite union of polyhedral regions. It is easy to see that the number of regions comprising these unions can be exceedingly large. Nevertheless, \cref{lemma:local} suggests that to check whether some point is an element of any node's solution graph, only those regions local to the point need be considered. This is the key concept that will drive the computational schemes developed in this section. 

Throughout this section, it will be assumed that the following routines are available for use:
\begin{itemize}
    \item A solver capable of efficiently finding a solution to an optimization problem of the form (\ref{eq:basic_opt}) in which $f$ is a convex quadratic function, and $C$ is a closed, non-empty polyhedron. One such freely available solver is \cite{osqp}.
    \item A computational routine capable of transforming between $H$- and $V$- representations of polyhedra. One such freely available routine is \cite{fukuda2003cddlib}.
    \item A solver capable of efficiently finding a solution to a Linear Mixed Complementarity Problem (LMCP), i.e. a problem of the form \begin{equation*}
        \begin{aligned}
            \mathrm{find} \ \ &  \mathbf{z} \\
            \mathrm{s.t.} \ \ & \exists \ \mathbf{w}_1, \mathbf{w}_2 : \\ 
            &\mathbf{M}\mathbf{z} + \mathbf{q} = \mathbf{w}_1 - \mathbf{w}_2 \\
            &\mathbf{0} \le \mathbf{z} - \mathbf{l} \perp \mathbf{w}_1 \ge 0 \\
            &\mathbf{0} \le \mathbf{u} - \mathbf{z} \perp \mathbf{w}_2 \ge 0,
        \end{aligned}
    \end{equation*}
    where $\mathbf{M}$ is a square matrix, and $\mathbf{l}$ and $\mathbf{u}$ are (potentially infinite) bounds on the variables $\mathbf{z}$. One such available solver is \cite{dirkse1995path}.
\end{itemize}

Using these foundational routines, \cref{alg:check_qp,alg:qp_solgraph,alg:qpn_solgraph,alg:equilibrium} are developed and then composed in \cref{alg:eq_search} as a routine for computing equilibrium points of a QPN.

\begin{algorithm}
\caption{Check QP Solution}\label{alg:check_qp}
\begin{algorithmic}[1]
\Require \textbf{ } \begin{itemize} \itemindent=-6pt
\item[$\bullet$] QP:
    \begin{itemize} \itemindent=-16pt
        \item convex cost $f(\x) = \frac{1}{2} \x^\intercal \mathbf{Q} \x + \x^\intercal \mathbf{q}$
        \item feasible set $C = \{ \x: \mathbf{A} \x + \mathbf{b} \ge \mathbf{0}\}$
        \item decision indices $J$
    \end{itemize}
    \item[$\bullet$] Candidate $\x \in \mathbb{R}^n$
\end{itemize}
\If{$\exists i : (\mathbf{A} \x + \mathbf{b})_i < 0$} 
\State \textbf{return false}
\EndIf
\State $I = \{ i : (\mathbf{A} \x + \mathbf{b})_i = 0 \}$
\State $\mathbf{Q}_J, \mathbf{q}_J \gets $ rows of $\mathbf{Q}$, $\mathbf{q}$ indexed by $J$
\State $\mathbf{A}_{I,J} \gets$ rows indexed by $I$, columns indexed by $J$
\State $\tilde{\mathbf{q}}_J = \mathbf{Q}_J \x + \mathbf{q}_J$
\State $\begin{aligned} \mathbf{\lambda}^* =  \argmin_{\mathbf{\lambda} \ge \mathbf{0}} ((\mathbf{A}_{I,J})^\intercal \mathbf{\lambda} - \tilde{\mathbf{q}}_J)^\intercal ((\mathbf{A}_{I,J})^\intercal \mathbf{\lambda} - \tilde{\mathbf{q}}_J)  \end{aligned}$ (QP solve, using e.g. \cite{osqp})
\If{$(\mathbf{A}_{I,J})^\intercal \mathbf{\lambda}^* = \tilde{\mathbf{q}}_J$}
\State \textbf{return true}, $\mathbf{\lambda}^*$
\Else{}
\State \textbf{return false}
\EndIf
\end{algorithmic}
\end{algorithm}

\begin{algorithm}
\caption{Generate Local Solution Graph for QP}\label{alg:qp_solgraph}
\begin{algorithmic}[1]
\Require  \textbf{ } \begin{itemize}\itemindent=-6pt
\item[$\bullet$] QP:
    \begin{itemize}\itemindent=-16pt
        \item cost $f(\x) = \frac{1}{2} \x^\intercal \mathbf{Q} \x + \x^\intercal \mathbf{q}$
        \item feasible set $C = \{ \x: \mathbf{A} \x + \mathbf{b} \ge \mathbf{0}\}$, $\mathbf{b} \in \R{m}$
        \item decision indices $J$
    \end{itemize}
    \item[$\bullet$] Reference solution $\x^* \in \mathbb{R}^n$
\end{itemize}
\State assert $\x$ is a solution for QP, get $\mathbf{\lambda}^*$ (\cref{alg:check_qp})
\State $I_p = \{ i : (\mathbf{A} \x^* + \mathbf{b})_i = 0 \}$
\State $I_d = \{ i: \mathbf{\lambda}^*_i = 0 \}$, $I_d' = [m] \setminus I_d$
\State $I_\mathrm{strong} = I_p \cap I_d'$ 
\State $I_\mathrm{weak} = I_p \cap I_d$
\State $\tilde{S} \gets $ empty collection of polyhedral solution graph regions
\For{$\hat{I} \in\mathcal{P}(I_\mathrm{weak}) $}
\State $I_a = I_\mathrm{strong} \cap \hat{I}$, $I_a' = [m] \setminus I_a$ 
\State $\mathbf{Q}_J, (\mathbf{A}^\intercal)_J \gets $ rows of $\mathbf{Q}, \mathbf{A}^\intercal$ indexed by $J$
\State $\mathbf{A}_{I_a}, \mathbf{b}_{I_a}, \mathbf{\lambda}_{I_a},\mathbf{A}_{I_a'}, \mathbf{b}_{I_a'}, \mathbf{\lambda}_{I_a'} \gets $ rows of $\mathbf{A}, \mathbf{b}, \mathbf{\lambda}$ indexed by $I_a, I_a'$
\State $ H = \left\{  \begin{aligned}  
\x \in \R{n}, \mathbf{\lambda} \in \R{m} : \ & \mathbf{Q}_J \x = (\mathbf{A}^\intercal)_J \mathbf{\lambda} \\
&\mathbf{A}_{I_a}\x + \mathbf{b}_{I_a} = \mathbf{0} \\
&\mathbf{A}_{I_a'}\x + \mathbf{b}_{I_a'} \ge \mathbf{0} \\
&\mathbf{\lambda}_{I_a} \ge \mathbf{0},  \ \mathbf{\lambda}_{I_a'} = \mathbf{0}
\end{aligned}\right\}$
\State $V, R \gets $ vertices, rays of $H$ (vertex enumeration, using e.g. \cite{fukuda2003cddlib})
\State $V_\x, R_\x \gets$ vertices, rays projected from $\R{n} \times \R{m}$ to $\R{n}$
\State $H_\x \gets $ halfspace representation of $V_\x, R_\x$ (using e.g. \cite{fukuda2003cddlib})
\State $\tilde{S} = 
\tilde{S} \cup H_\x$
\EndFor
\State \textbf{return $S$}
\end{algorithmic}
\end{algorithm}

\Cref{alg:check_qp} is a procedure for validating that some vector $\x \in \R{n}$ is a solution to a QP defined by the cost function $f$, feasible set $C$, and decision indices $J$. Since $J \subset [n]$, in general, some elements of $\x$ will act as parameters to the QP. This algorithm fixes the primal variables $\x_J$ (elements of $\x$ indexed by $J$), and tries to identify dual variables $\lambda$ which satisfy the necessary and sufficient conditions of optimality (\ref{eq:sol_graph_QP}). If dual variables satisfying these conditions are found, they are returned as a certificate of optimality. Otherwise, it is determined that the vector $\x$ is not in the solution graph of the given QP. 

\Cref{alg:qp_solgraph} produces a set $\tilde{S}$ which approximates the full solution graph $S$ of a QP in a neighborhood around a supplied reference $\x^*$, following \cref{lemma:QP_sol}. Specifically, the computed set $\tilde{S}$ satisfied $\tilde{S} \cap \mathcal{N}_\epsilon(\x^*) = S \cap \mathcal{N}_\epsilon(\x^*)$ for some $\epsilon > 0$. This algorithm proceeds by first identifying (for the solution $\x^*$) which affine inequality constraints are \emph{strongly} active, \emph{weakly} active, or inactive, where active means the affine expression achieves its lower bound of zero. A strongly active constraint has a corresponding dual variable with a strictly positive value, whereas a weakly active constraint has an associated dual variable with value of zero. 

For every subset of weakly active constraint indices, the constraints are partitioned into two new sets. One which include the strongly active constraints and the subset of weakly active constraints, and another which contains all other constraints. This partition leads to a polyhedral region in the joint primal-dual variable space in which every element satisfies the optimality conditions for the QP with a particular assignment for the complementarity conditions. When projected from the primal-dual space into the primal space, these sets result in a polyhedral region of the QP solution graph that contains $\x$. Hence, all local polyhedral regions are found by considering all subsets of the weakly-active constraints, and then projecting the resulting polyhedral regions by first converting them to vertex representation, projecting all vertices, and then converting back to halfspace representation.

\Cref{alg:qpn_solgraph} leverages \cref{alg:qp_solgraph} to compute local representation (with respect to a reference solution $\x^*$) of the solution graph of a QP node $i$ within a QPN, following \cref{cor:sol_union_local}. It is assumed that local representations of the solution graphs for all children of node $i$ are known and provided. By \cref{lemma:QPN_sol}, the solution graphs for all nodes in a QPN, and therefore the children of node $i$, are unions of not-necessarily closed polyhedral regions. The construction of node $i$'s solution graph proceeds by considering the intersection of node $i$'s feasible set $C^i$ with the intersection of all child solution graph representations, which results in another union of polyhedral regions. For every resulting region $C^l$ (with $l$ an index enumerating these sets), the solution graph of node $i$'s QP using $C^l$ as the feasible set is found using \cref{alg:qp_solgraph}. A set $Z^{i,l}$ is then formed by taking the union of this resulting local solution graph with the complement of the set $C^l$. When the sets $Z^{i,l}$ are intersected, a set is produced which, when intersected with $\mathcal{N}_\epsilon(\x^*)$, is equivalent to those in (\ref{eq:sol_union_local}) as required. For purposes of searching for equilibrium points of QPNs, only those regions which are local to $\x^*$ are returned.

\begin{algorithm}
\caption{Generate Local Solution Graph for a QPN Node}\label{alg:qpn_solgraph}
\begin{algorithmic}[1]
\Require  \textbf{ } \begin{itemize}\itemindent=-6pt
\item[$\bullet$] QP Node $i$:
    \begin{itemize}\itemindent=-16pt
        \item cost $f^i(\x) = \frac{1}{2} \x^\intercal \mathbf{Q}^i \x + \x^\intercal \mathbf{q}^i$
        \item feasible set $C^i = \{ \x: \mathbf{A}^i \x + \mathbf{b}^i \ge \mathbf{0}\}$
        \item decision indices $J^i$
    \end{itemize}
    \item[$\bullet$] Reference solution $\x^* \in \mathbb{R}^n$
    \item[$\bullet$] Local solution graphs for child nodes: not-necessarily-closed polyhedral regions $S^{j,k}$, where $\mathcal{N}_\epsilon(\x^*) \cap S^{j} = \mathcal{N}_\epsilon(\x^*) \cap \bigcup_{k \in K^j} S^{j,k}$, for all $j$ such that $(i,j)\in E$.
    \item[] ($E$ are network edges, $K^j$ are index sets, and $\x^* \in S^{j,k}$ for all $j,k$).
\end{itemize}
\State $I \gets J^i \bigcup_{j: (i,j) \in R} J^j$ ($R$ is the  set of reachable transitions for network)
\State $l \gets 0$
\For{$\mathbf{k} \in \prod_{j: (i,j) \in E} K^{j}$}
\State $l = l + 1$
\State $C^l = C^i \bigcap_{j:(i,j) \in E} \closed{S^{j,\mathbf{k}_j}}$
\State $S^{i,l} \gets$ solution graph for QP ($f^i$, $C^l$, $I$) local to $\x^*$ (\cref{alg:qp_solgraph})
\State $Z^{i,l} = S^{i,l} \cup (C^l)'$
\EndFor
\State $\tilde{S}^i = \bigcap_{l} Z^{i,l}$
\State $\tilde{S}^i := \bigcup_{k\in K^i} S^{i,k}$ (can be expressed as a union of polyhedral regions)
\State $\tilde{S}^i \gets \bigcup_{k: \x^* \in \closed{S^{i,k}}} S^{i,k}$ (exclude all non-local regions).
\State \textbf{Return $S^i$}
\end{algorithmic}
\end{algorithm}

\begin{algorithm}
\caption{Compute Nash Equilibrium among QPs}\label{alg:equilibrium}
\begin{algorithmic}[1]
    \Require  \textbf{ } \begin{itemize}\itemindent=-6pt
    \item[$\bullet$] QPs $\{\mathrm{QP}^1,...,\mathrm{QP}^N\}$. 
    \item[] $\mathrm{QP}^i$: 
        \begin{itemize}\itemindent=-16pt
            \item cost $f^i(\x) = \frac{1}{2} \x^\intercal \mathbf{Q}^i \x + \x^\intercal \mathbf{q}^i$
            \item feasible set $C^i = \{ \x: \mathbf{A}^i \x + \mathbf{b}^i \ge \mathbf{0}\}$, $\mathbf{A}^i \in \R{m^i \times n}$
            \item decision indices $J^i$, $|J^i| = n^i \le n$
        \end{itemize}
        \item candidate $\x \in \R{n}$
    \end{itemize}
    
    \State $J = \bigcup_{i \in [N]} J^i$, $J' = [n]\setminus J$
    \State $\tilde{n} = \sum_{i=1}^N n^i$, $\hat{n} = | J|$, $m = \sum_{i=1}^N m^i$
    
    \State $\mathbf{Q} = \begin{bmatrix} 
    \mathbf{Q}^1_{J^1, J} \\
    \vdots \\
    \mathbf{Q}^N_{J^N, J}
    \end{bmatrix}$, 
    $\mathbf{R} = \begin{bmatrix} 
    \mathbf{Q}^1_{J^1, J'} \\
    \vdots \\
    \mathbf{Q}^N_{J^N, J'}
    \end{bmatrix}$, 
    $\mathbf{q} = \begin{bmatrix} 
    \mathbf{q}^1_{J^1} \\
    \vdots \\
    \mathbf{q}^N_{J^N}
    \end{bmatrix}$
    
    \State $\mathbf{A} = \begin{bmatrix} 
    \mathbf{A}^1_{\cdot, J} \\
    \vdots \\
    \mathbf{A}^N_{\cdot, J}
    \end{bmatrix}$, 
    $\mathbf{\tilde{A}} = \begin{bmatrix} 
    \mathbf{A}^1_{J^1, J} & & \\
    & \ddots &\\
    & & \mathbf{A}^N_{J^N, J}
    \end{bmatrix}$
    $\mathbf{B} = \begin{bmatrix} 
    \mathbf{A}^1_{\cdot, J'} \\
    \vdots \\
    \mathbf{A}^N_{\cdot, J'}
    \end{bmatrix}$, $\mathbf{b} = \begin{bmatrix} 
    \mathbf{b}^1 \\
    \vdots \\
    \mathbf{b}^N
    \end{bmatrix}$
    
    \State $\mathbf{l} = \begin{bmatrix}
        -\mathbf{\infty}_{\hat{n}} \\ 
        -\mathbf{\infty}_{\tilde{n}} \\ 
        \mathbf{0}_m
    \end{bmatrix}$, $\mathbf{u} = \begin{bmatrix}
        +\mathbf{\infty}_{\hat{n}} \\ 
        +\mathbf{\infty}_{\tilde{n}} \\ 
        \mathbf{0}_m
    \end{bmatrix}$
    
    \State $\mathbf{M} = 
    \begin{bmatrix} 
    \mathbf{0}_{\hat{n},\hat{n}} & \mathbf{0}_{\hat{n},\tilde{n}} & \mathbf{0}_{\hat{n},m} \\ 
    \mathbf{Q} & \mathbf{0}_{\tilde{n}, \tilde{n}} & -\mathbf{\tilde{A}}^\intercal \\
    \mathbf{A} & \mathbf{0}_{m,\tilde{n}} & \mathbf{0}_{m, m}
    \end{bmatrix}
    $, $\mathbf{N} = 
    \begin{bmatrix} 
     \mathbf{0}_{\hat{n},n-\hat{n}}\\
    \mathbf{R} \\
    \mathbf{B}
    \end{bmatrix}
    $, $\mathbf{o} = 
    \begin{bmatrix} 
    \mathbf{0}_{\hat{n}} \\
    \mathbf{q} \\
    \mathbf{b}
    \end{bmatrix}$
    
    \State LMCP $ \gets (\mathbf{M}, \mathbf{N}\x_{J'} + \mathbf{o}, \mathbf{l}, \mathbf{u})$
    \State $(\mathbf{\psi}, \x_{J}, \mathbf{\lambda}) \gets $ solve LMCP (e.g. using \cite{dirkse1995path})
    \State \textbf{Return} $\x$
\end{algorithmic}
\end{algorithm}

In \cref{alg:equilibrium}, a procedure is given for computing an equilibrium among an unconnected collection of $N$ QPs. These QPs each have a cost function, feasible set, and set of decision indices, as in standard QPNs. The decision indices need not be unique. It may be that some index $j \in J^{i_1} \cap J^{i_2}$ for two indices $i_1,i_2 \in [N]$. The algorithm simply constructs a large LMCP from the optimality conditions for each of the $N$ QPs. Because the decision indices for each of the QPs are not necessarily unique, slack variables must be introduced to ensure that the matrix $\mathbf{M}$ of the LMCP remains square. In this formulation, the slack variables do not enter the optimality conditions for any of the QPs, which means that a solution to the LMCP will not exist unless there exist values of the shared decision variables which are simultaneously optimal for all players who share those decision indices. As is discussed below, this is not an issue for the situation most commonly encountered for QPNs. A thorough treatment of existence criteria for an equilibrium point to exist for these type of problems can be found in \cite{facchinei2003finite}. 

Finally, \cref{alg:eq_search} can be described, which is the main routine for computing equilibrium points of a QPN. This algorithm starts with some initial guess $\x$, and then works from the deepest nodes of the network upward, repeatedly constructing local representations of each node's solution graph, or modifying the iterate $\x$ so that it is indeed a solution to the given node's optimization problem, and then restarting the solution graph construction from the bottom nodes. This continues until a point is found which lies within the solution graphs of every node in the graph, and is hence an equilibrium point. 

\begin{algorithm}
\caption{QPN Equilibrium Search}\label{alg:eq_search}
\begin{algorithmic}[1]
\Require  \textbf{ } \begin{itemize}\itemindent=-6pt
\item[$\bullet$] QPN:
    \begin{itemize}\itemindent=-16pt
        \item Set of QP nodes $\{f^i, C^i, J^i\}_{i \in [N]}$
        \item Network Edges $E$
    \end{itemize}
    \item[$\bullet$] Initialization $\x \in \mathbb{R}^n$
    \item[$\bullet$] Topological depth mapping $\mathcal{L} = \{L_1, ..., L_D\}$
\end{itemize}
\For{$d \in \{D,...,1\}$} \label{line:restart}
    \State agreement $\gets$ \textbf{true}
    \For{$i \in L_d$}
        \State $C = C^i \bigcap_{j: (i,j) \in E} S^j_\epsilon(\x)$ (simply $C^i$ if $d=D$)
        \State Let $C := \bigcup_{l} C^{i,l}$,  ($C^{i,l}$ polyhedral)
        \If{for all $l$, $\x$ solves $QP(f^i, \closed{C^{i,l}}, J^i)$ (\cref{alg:check_qp})}
            \State $S^i_\epsilon(\x) \gets$ local solution graph $(\{f^i, C^i, J^i\}, \x, \{S^j_\epsilon(\x)\}_{j: (i,j) \in E})$ (\cref{alg:qpn_solgraph})
            \State $l_i \gets 1$
        \Else{}
            \State agreement $\gets$ \textbf{false}
            \State $l_i \gets l$ for which $\x$ does not solve $QP(f^i, C^{i,l}, J^i)$
        \EndIf
    \EndFor
    \If{agreement = \textbf{false}}
        \State $I^i = \bigcup_{j \in D^i} J^j$ ($D^i$ as in \cref{eq:Di})
        \State $\x \gets$ equilibrium for $\{ QP(f^i, C^{i,l_i}, I^i) : i \in L_d\}$ (\cref{alg:equilibrium})
        \State $d \gets D$, \textbf{go to} \cref{line:restart} ($D$ as in depth of depth-mapping)
    \EndIf
\EndFor
\State \textbf{Return $\x$} 
\end{algorithmic}
\end{algorithm}

In order to define this bottom-up procedure, a \emph{topological depth mapping} for the network must be created. A topological depth mapping $\mathcal{L}$ over a QPN (more generally, a MPN) is defined as a ordered set of layers $\{L_1, ..., L_D\}$, where each layer is a set of nodes. The depth map is topological in that for all $i \in L_d$, if $(i,j) \in E$ (with $E$ being the set of edges in the network), then $j \in L_{d'}$ for some $d' > d$. Any such depth map is said to have depth $D$. Note that a multiple topological depth maps may exist for the same network, as shown in \cref{fig:depth_maps}. 

\begin{figure}[hbt]
    \centering
    \includegraphics[width=0.55\linewidth]{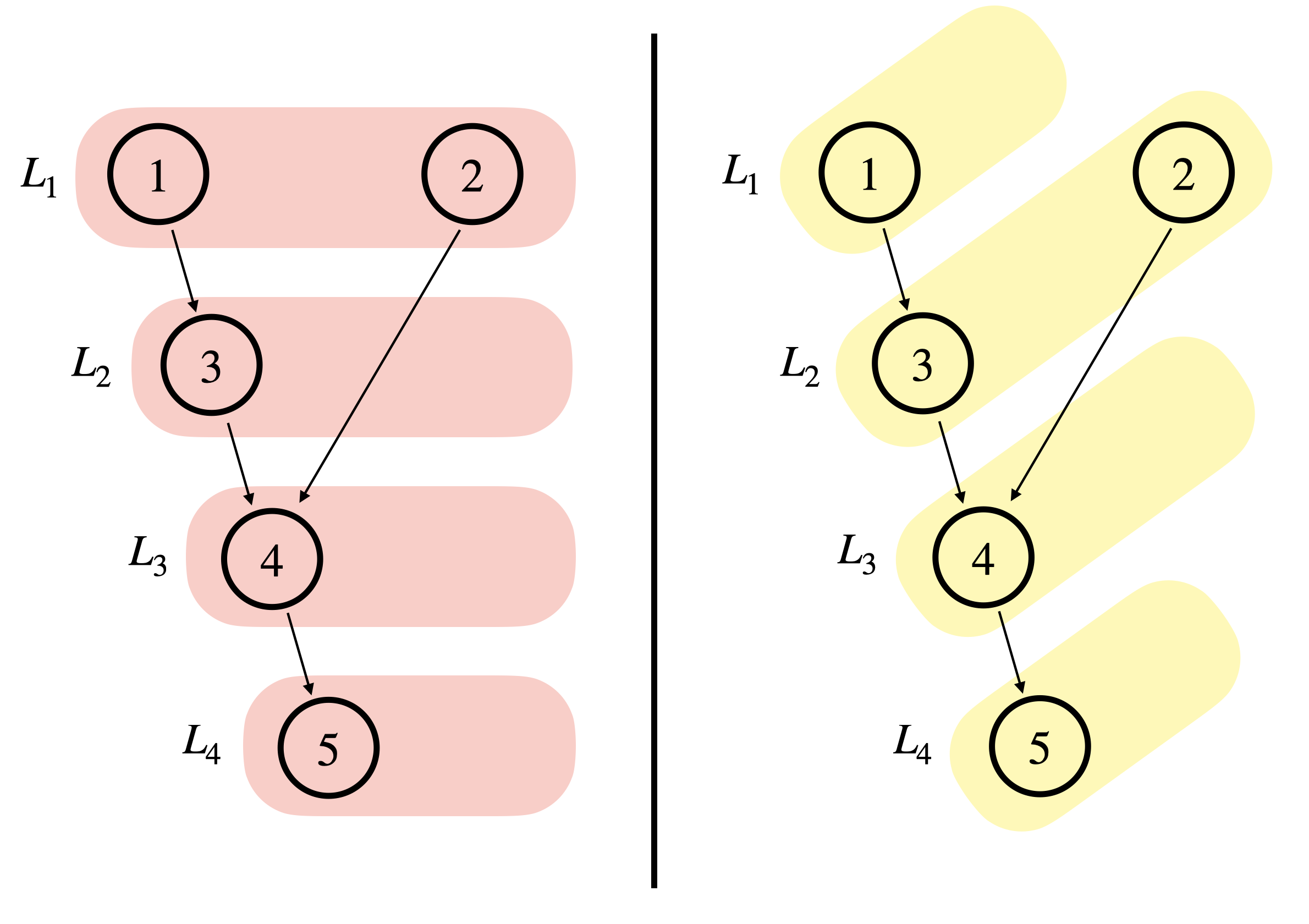}
    \caption{Two valid depth mappings for the same network.}
    \label{fig:depth_maps}
\end{figure}

Given a valid depth map $\mathcal{L}$ and an initial iterate $\x$, the procedure starts at the deepest layer $L_D$, and checks if the current iterate is a solution for the nodes in that layer or not. If it is a solution for all such nodes, then the solution graphs for each node are generated using \cref{alg:qpn_solgraph}. If not, then an equilibrium point is found for the nodes in that layer using \cref{alg:equilibrium}. Note that doing so only modifies the subset of the iterate $\x$ corresponding to the decision nodes in this deepest layer. After an equilibrium is found, the algorithm restarts, and the local solution graphs are computed for the nodes at this layer. Then, working upwards in the depth mapping, the procedure repeats. For a given layer, and for each node in the layer, the combined feasible set is generated by intersecting the node's feasible set with the intersection of local solution graphs for each of its children (which by definition, were constructed in deeper layers). These combined feasible sets are again unions of not-necessarily closed polyhedral regions. Following \cref{lem:unions}, the iterate $\x$ is checked for node optimality using the closure of each of the polyhedral regions. 

If the iterate $\x$ solves each QP node using each of the polyhedral pieces as its feasible set, then the local solution graphs for each of these nodes are constructed using \cref{alg:qpn_solgraph}, and the algorithm proceeds to the next layer above. Alternatively, if at least one of the nodes are not satisfied with the iterate, then an equilibrium must be found among the nodes at the current layer. To do this, a particular polyhedral region from the combined feasible set needs to be used for each node. The first region which results in $\x$ not being an optimum for the given node is chosen. Nodes for which $\x$ is an optimum choose such a region arbitrarily. In practice, it is possible that this procedure for choosing regions results in inconsistency. 

For example consider two nodes which share a single child node. Say that the child node solution graph consists of two regions, region $A$ and region $B$, which intersect only at the point $\x$. If $\x$ is not an optimum for node $1$ when using region $A$, and $\x$ is not an optimum for node $2$ when using region $B$, then attempting to compute an equilibrium for node $1$ with region $A$ and node $2$ with region $B$ will result in failure. Therefore, additional care must be taken to avoid such inconsistencies. This can be accomplished through careful bookkeeping, but such a procedure is omitted from \cref{alg:eq_search} for purposes of presentation. 

If an equilibrium need be computed for a given layer, then the iterate $\x$ is updated to satisfy the equilibrium conditions, and then the algorithm returns to the deepest layer $L_D$ to reconstruct representations of the solution graphs of all nodes which are valid locally to the new iterate. If an equilibrium computation is performed at some depth $d$, then by construction, the new iterate will satisfy optimality conditions for all nodes at all depths $\{D,...,d+1\}$. 

The algorithm presented follows directly from the results in \cref{sec:dev}. As will be seen in the following section, it is effective at finding equilibrium points for moderate sized, non-trivial QPNs. However, it is not without shortcomings. Significant amounts of computations can be wasted generating solution graphs of low-level nodes, only for the iterate to not satisfy optimality conditions at some high-level node. The number of polyhedral regions comprising the solution graphs of the nodes can also become immense, especially for deep QPNs. This problem may be avoidable in some instances. At an equilibrium point, all nodes in a QPN may only have a single polyhedral region in their local solution graphs. However, along the path of iterates encountered by \cref{alg:eq_search}, there may be points which result in local solution graphs with so many regions that computation is rendered completely intractable. 

Finally, there are known failure modes to the algorithm described. It is possible that cycling occurs, leading to an infinite execution. Practical implementations need to monitor for cycling and report failure to ensure the computation terminates. Furthermore, the success of the algorithm depends on successful returns from the routines described in the beginning of this section. However, numerical conditioning and floating point errors can result in some of the computations failing or providing inaccurate results, which propagate into larger errors for \cref{alg:eq_search}. 

An open-source Julia package implementing the algorithms presented in this section was created as a part of this work \cite{Laine_QPNets_jl_2024}. In the section to follow, examples of interesting QPNs are presented, and they are solved using the methods above.

\section{Examples}
\label{sec:comp_examples}
Several examples of MPNs are presented in this section. Specifically, \cref{example:ridge} is used to explicitly illustrate each of the steps in \cref{alg:eq_search}. \Cref{example:constellation_game} is used to explore how many different network configurations can be applied to the same set of mathematical program nodes, each defining a different MPN and corresponding set of equilibrium points. Finally, \cref{example:robust_avoid} is used to demonstrate that interesting problems can be naturally posed in the MPN framework.

The examples in this section are described using a tabular format designed to succinctly define the network configuration and properties of the constituent nodes. This description begins by describing the decision variables $\x$, and when useful, assigning context-specific variable names to subsets of the variables. When the network configuration is made explicit, a visual depiction of it is listed following the decision variable description. Finally, the cost function, feasible set, and set of private decision variables are listed for each of the nodes in the network.

Code for setting up each of these examples and solving them is provided in the package \cite{Laine_QPNets_jl_2024}.

\subsection{Simple Bilevel}

\begin{figure}[htb]
    \centering
    \includegraphics[width=0.95\linewidth]{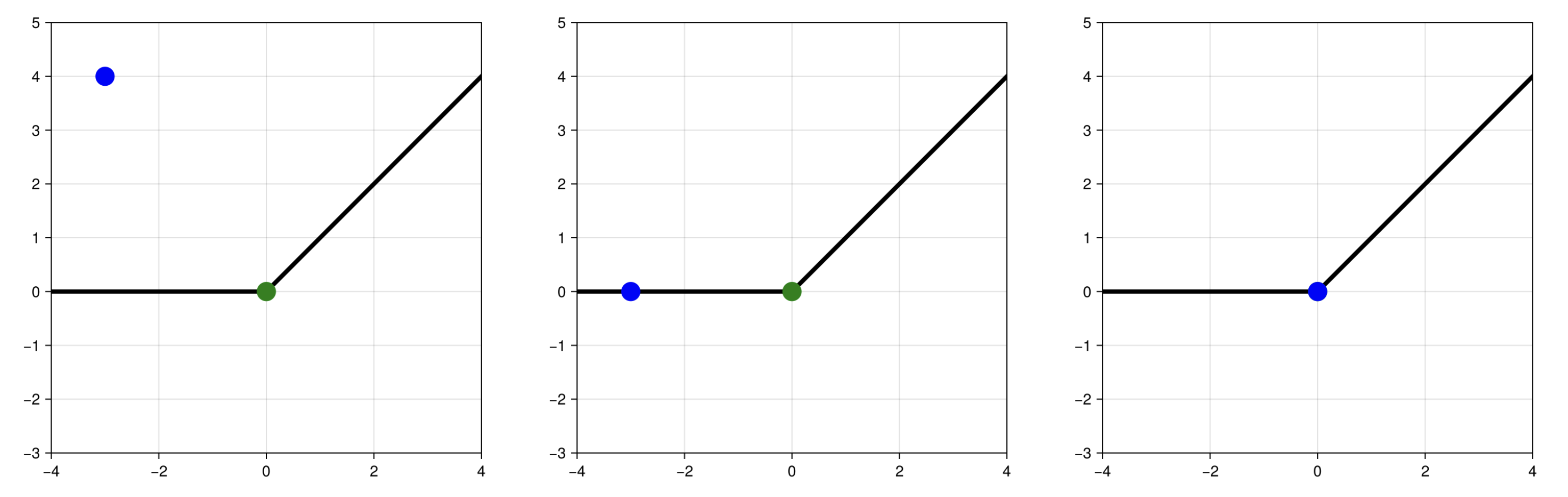}
    \caption{A depiction of the iterates (shown from left to right) that are encountered when solving for an equilibrium of \cref{example:ridge} with $\x$ initialized to $[0,0,-3,4]$.}
    \label{fig:ridge}
\end{figure}

\vspace{8pt}
\noindent\begin{minipage}\textwidth
\begin{example}[Simple Bilevel] \label{example:ridge}
    \renewcommand{\arraystretch}{1.6}
    \begin{center}
    \begin{tabular}{ |c c|  }
     \hline
     Dec. vars: & $\mathbf{x} := \left[ x_1, x_2, x_3, x_4 \right] \in \R{4}$ \\[2pt]
     Network: &
         \begin{tikzpicture}[scale=0.75,baseline=(current bounding box.center), every node/.style={circle, draw, minimum size=0.8cm,font=\scriptsize, line width=1.0}]
            \node[circle, draw] (1) at (-1.0,-0.1) {1};
            \node[circle, draw] (2) at (1.0,-0.1) {2};
            \node[draw=none] (space) at (0.0,-0.3) {};
            \graph[edges={line width=1pt}] {(1) -> (2)}; 
        \end{tikzpicture}
    \\[5pt] \hline \hline
     \multicolumn{2}{|c|}{ Node $1$: } \\\hline
     Cost: & $\frac{1}{2}(x_3-x_1)^2 + \frac{1}{2}(x_4-x_2)^2$ \\
     Feasible set: & $\{ \mathbf{x} \in \R{4} \}$ \\
     Private vars: & $x_3$ \\
     Child nodes: & \{2\} \\
     \hline
    \multicolumn{2}{|c|}{ Node $2$: } \\\hline
    Cost: & $\frac{1}{2}(x_4 - x_3)^2$ \\
     Feasible set: & $\{ \mathbf{x}: x_4 \ge 0 \}$ \\
     Private vars: & $x_4$ \\
     Child nodes: & $\{\}$ \\
     \hline
    \end{tabular}
    \end{center}
\end{example}
\vspace{\parskip}
\end{minipage}

This example describes a small bilevel quadratic program which, despite its simplicity, is useful for understanding the computational routines of \cref{sec:comp}. Since the variables $x_1$ and $x_2$ are components of $\x$ but do not appear as private variables of either node 1 or node 2, they act as parameters to the problem, and the solution graphs for each node and the resulting set of equilibrium points for the network will be defined in terms of the values of these parameters. In other words, for different initializations of these variables, different equilibrium points will result. Here, a trace of \cref{alg:eq_search} is explored using the initial point $\x = [0,0,-3,4]$. 

There is a unique depth mapping associated with this network, 
\begin{equation*}
    \mathcal{L} = \{L_1, L_2\} = \{\{1\}, \{2\}\}.
\end{equation*}
Starting with $d=2$, the initial point $\x$ is checked for optimality of node $2$. There is a single constraint in $C^2$, and it is inactive at the point $\x$. Therefore, the point can only be optimal if $\tilde{q}_J = 0$ in \cref{alg:check_qp}.  However, at $\x$, $\tilde{q}_J = 7$, and therefore it is identified that $\x$ is not in the solution graph for this node, and an equilibrium must be found. As node $2$ is the only node in layer $L_2$, this results in simply updating $x_4$ to be optimal for node $2$, i.e. $x_4 \gets 0$. The algorithm then proceeds to verify that $\x = [0,0,-3,0]$ is indeed a solution for node $2$, identifying a value of $\lambda^* = 2$.

With the new iterate $\x$ identified as a solution for all nodes in $L_2$, local representations of the solution graphs are computed. For node $2$, there is a single strongly active constraint, resulting in a single polyhedral region:
\begin{align*}
\left\{
\begin{aligned}
 \x, \mathbf{\lambda} : \ &x_4 - x_3 = \lambda, \\
    &x_4 = 0, \\
    &\lambda \ge 0.   
\end{aligned}
    \right\}
\end{align*}
Projecting this region from primal-dual space to primal space results in the local representation of the solution graph
\begin{align*}
    S^{2,1} = \left\{  \x : x_4=0, x_3 \le 0\right\}.
\end{align*}
At this point, the algorithm proceeds to level $L_1$, which contains only node $1$. Here, optimality is checked when considering the local solution graph of node $2$ as part of the feasible set for node $1$. The current iterate is not optimal. The only optimal point when considering the local representation of node $2$'s solution graph is $[0,0,0,0]$, and as such the iterate is updated to this point. 

With $\x$ updated, the algorithm returns to $L_2$ to update the representation of the solution graph for node $2$ in the vicinity of the new iterate. Now the constraint for node $2$ is weakly active, rather than strongly active. After projecting, this results in two components of the local solution graph: $S^{2,1}$ (as before) and 
\begin{align*}
    S^{2,2} = \left\{  \x : x_4=x_3, x_3 \ge 0\right\}.
\end{align*}
These two components are shown in \cref{fig:ridge} as the thick black lines. Returning to $L_1$, node $1$ is checked for optimality when using both of components of the solution graph for node $2$. The first was previously used to arrive at the current iterate, so $\x$ is indeed a solution when using that component. It is determined that $\x$ is also an optimum when considering the second component, and an equilibrium for the QPN has hence been identified. 

At this point, the algorithm could terminate without computing the solution graph for node $1$. However, for illustration purposes, its construction is documented here. The solution graph for node $1$ is first computed using the region $S^{2,1}$ as the feasible set, then using $S^{2,2}$. In both instances, both $x_3$ and $x_4$ are the decision variables. Using $S^{2,1}$ results in $S^{1,1}$, and $S^{2,2}$ results in $S^{1,2}$. Both of these sets have two polyhedral components, and are given by the following:
\begin{align*}
    S^{1,1} &= \left\{ \begin{aligned}
        \x : \ & x_3 = x_1 \\
            & x_4 = 0 \\
            & x_1 \le 0
    \end{aligned} \right\} \cup 
    \left\{ \begin{aligned}
        \x : \ & x_3 = 0 \\
            & x_4 = 0 \\
            & x_1 \ge 0
    \end{aligned} \right\} \\
    S^{1,2} &= \left\{ \begin{aligned}
        \x : \ & x_3 = 0.5(x_1+x_2) \\
            & x_4 = 0.5(x_1+x_2) \\
            & x_1+x_2 \ge 0
    \end{aligned} \right\} \cup 
    \left\{ \begin{aligned}
        \x : \ & x_3 = 0 \\
            & x_4 = 0 \\
            & x_1+x_2 \le 0
    \end{aligned} \right\}.
\end{align*}

Following \cref{alg:qpn_solgraph}, the sets $Z^{1,1}$ and $Z^{1,2}$ are constructed from $S^{1,1}$, $S^{2,1}$, $S^{1,2}$, and $S^{2,2}$.

\begin{align*}
    Z^{1,1} &= S^{1,1} \cup
    \left\{ \x: \ x_3 > 0 \right\} \cup \left\{ \x: \ x_4 > 0 \right\} \cup \left\{ \x: \ x_4 < 0 \right\}
    \\
    Z^{1,2} &= S^{1,2} \cup  \left\{ \x: \ x_4 > x_3 \right\} \cup \left\{ \x: \ x_4 < x_3 \right\} \cup \left\{ \x: \ x_3 < 0 \right\}.
\end{align*}

Finally, by intersecting $Z^{1,1} \cap Z^{1,2}$, and considering the non-emptpy sets whose closure includes $\x$, the local approximation for $S^1$ is given:

\begin{align*}
    S^1 = \left\{ \begin{aligned}
        \x: \ & x_3 = x_1 \\ & x_4 = 0 \\ &x_1 < 0
    \end{aligned} \right\} \cup \left\{ \begin{aligned}
        \x: \ & x_3 = 0.5(x_1+x_2) \\ &x_4 = 0.5(x_1+x_2) \\ &x_1+x_2 > 0
    \end{aligned} \right\} \cup \left\{ \begin{aligned}
        \x: \ & x_3 = 0 \\ &x_4 = 0 \\ &x_1 \ge 0 \\ &(x_1+x_2) \le 0
    \end{aligned} \right\}.
\end{align*}

This example highlights the essence of \cref{alg:eq_search}. Many of the steps become more involved for 

\subsection{Constellation Game}

\begin{figure}[htb]
    \centering
    \includegraphics[width=0.95\linewidth]{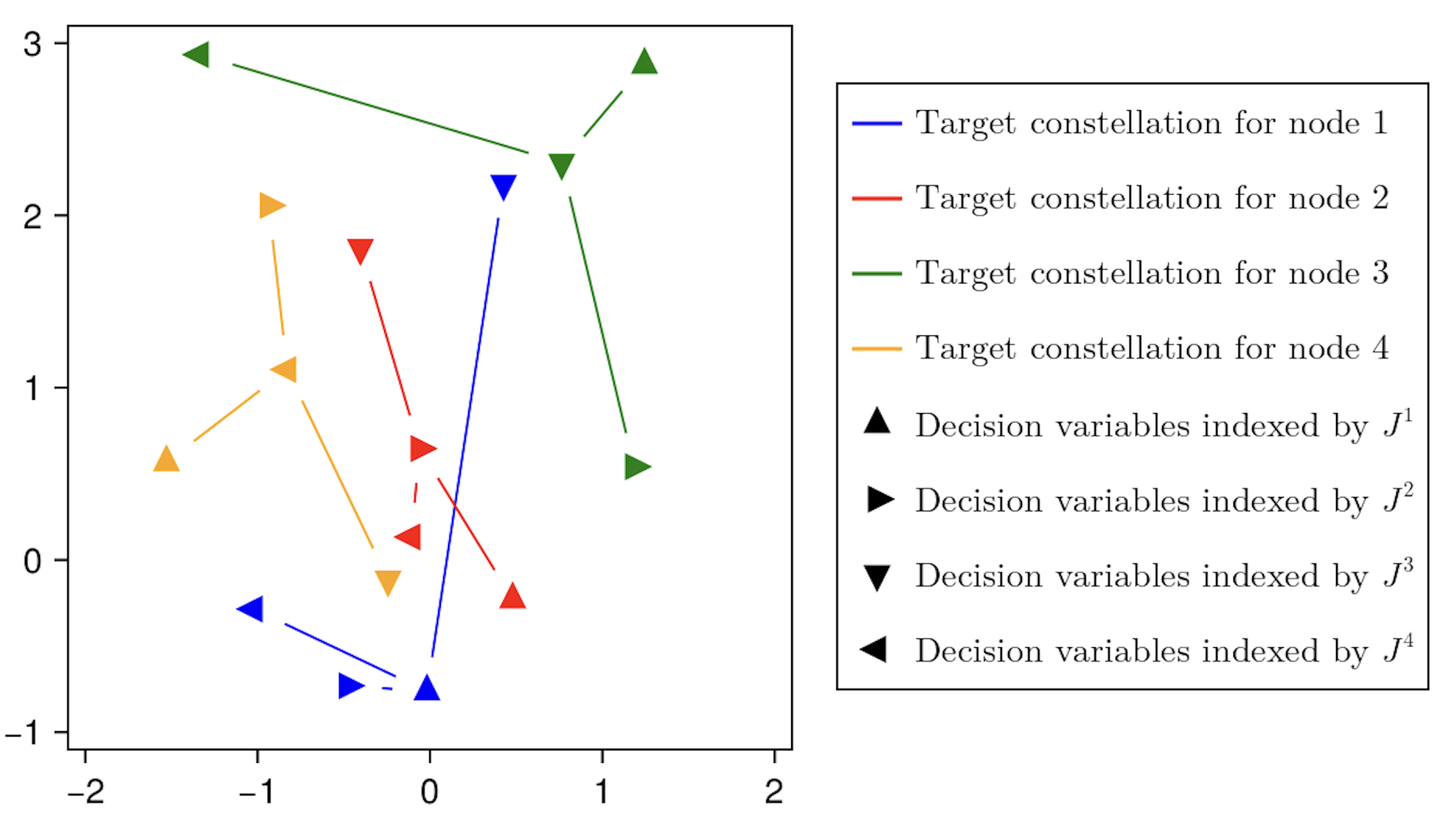}
    \caption{A depiction of each node's cost function in \cref{example:constellation_game} for randomly generated values of the parameters $\mathbf{g}^i$ and $\mathbf{r}^{i,j}$. Each constellation (depicted by a color) indicates how a particular node would decide $\mathbf{x}$ if it had unconstrained control over the entire vector. The differently-oriented triangles are used to indicate different sub-vectors of $\mathbf{x}$.
    }
    \label{fig:constellations}
\end{figure}

The following example is introduced to illustrate that the same collection of mathematical programs generally admit different equilibrium points when considered under different network configurations, and therefore each node in an MPN will have a preference on how the network is arranged.  

\vspace{8pt}
\noindent\begin{minipage}\textwidth
\begin{example}[Constellation Game] \label{example:constellation_game}
    \renewcommand{\arraystretch}{1.6}
    \begin{center}
    \begin{tabular}{ |c c|  }
     \hline
     Dec. vars: & $\mathbf{x} := \left[ \p^1, \p^2, \p^3, \p^4 \right], \ \p^i \in \R{2}$ \\[2pt] 
     Network: & See \cref{fig:four_node_example}.
    \\[5pt] \hline \hline
     \multicolumn{2}{|c|}{ Nodes $i \in \{1,2,3,4\}$: } \\\hline
     Cost: & $\|\p^i - \mathbf{g}^{i}\|_2^2 + \sum_{j=1,j\ne i}^4 \| {\p}^j - {\p}^i - \mathbf{r}^{i,j}\|_2^2$ \\
     Feasible set: & $\{ \mathbf{x}: \|\p^i\|_\infty \le 5 \}$ \\
     Private vars: & $\p^i$ \\
     Child nodes: & See \cref{fig:four_node_example}. \\
     \hline
    \end{tabular}
    \end{center}
\end{example} 
\vspace{\parskip}
\end{minipage}

\SetTblrInner{rowsep=2.9pt}
\SetTblrInner{colsep=3pt}

\begin{figure}[ht!]
\centering
\begin{tblr}{|c|c|c|c|c|}
    \hline 
\csname graph22\endcsname &
\csname graph21\endcsname &
\csname graph23\endcsname &
\csname graph43\endcsname &
\csname graph29\endcsname \\
(-6.24±0.08)\% &
(-5.87±0.07)\% &
(-5.72±0.10)\% &
(-5.48±0.06)\% &
(-5.40±0.11)\% \\\hline
\csname graph19\endcsname &
\csname graph40\endcsname &
\csname graph18\endcsname &
\csname graph37\endcsname &
\csname graph45\endcsname \\
(-5.37±0.10)\% &
(-5.19±0.11)\% &
(-5.07±0.05)\% &
(-4.83±0.11)\% &
(-4.72±0.10)\% \\\hline
\csname graph17\endcsname &
\csname graph46\endcsname &
\csname graph26\endcsname &
\csname graph34\endcsname &
\csname graph35\endcsname \\
(-4.64±0.04)\% &
(-4.45±0.10)\% &
(-4.34±0.09)\% &
(-4.10±0.09)\% &
(-4.08±0.09)\% \\\hline
\csname graph33\endcsname &
\csname graph24\endcsname &
\csname graph27\endcsname &
\csname graph44\endcsname &
\csname graph41\endcsname \\
(-3.70±0.09)\% &
(-3.69±0.07)\% &
(-3.34±0.08)\% &
(-3.31±0.06)\% &
(-3.12±0.08)\% \\\hline
\csname graph25\endcsname &
\csname graph47\endcsname &
\csname graph28\endcsname &
\csname graph32\endcsname &
\csname graph20\endcsname \\
(-2.85±0.05)\% &
(-2.73±0.07)\% &
(-2.53±0.06)\% &
(-2.52±0.06)\% &
(-2.43±0.04)\% \\\hline
\csname graph6\endcsname &
\csname graph39\endcsname &
\csname graph7\endcsname &
\csname graph30\endcsname &
\csname graph5\endcsname \\
(-2.40±0.04)\% &
(-2.29±0.07)\% &
(-2.07±0.05)\% &
(-2.07±0.05)\% &
(-1.95±0.03)\% \\\hline
\csname graph38\endcsname &
\csname graph36\endcsname &
\csname graph12\endcsname &
\csname graph13\endcsname &
\csname graph31\endcsname \\
(-1.88±0.06)\% &
(-1.86±0.06)\% &
(-1.86±0.06)\% &
(-1.84±0.06)\% &
(-1.63±0.05)\% \\\hline
\csname graph42\endcsname &
\csname graph10\endcsname &
\csname graph11\endcsname &
\csname graph9\endcsname &
\csname graph8\endcsname \\
(-1.42±0.05)\% &
(-1.42±0.05)\% &
(-1.41±0.05)\% &
(-1.19±0.03)\% &
(-1.17±0.03)\% \\\hline
\csname graph16\endcsname &
\csname graph15\endcsname &
\csname graph14\endcsname &
\csname graph2\endcsname &
\csname graph3\endcsname \\
(-0.98±0.04)\% &
(-0.98±0.04)\% &
(-0.98±0.04)\% &
(-0.70±0.02)\% &
(-0.49±0.03)\% \\\hline
\csname graph4\endcsname &
\csname graph1\endcsname & & & \\
(-0.48±0.03)\% & 0\% & & & \\\hline
\end{tblr}
\caption{The 47 unique network configurations for the constellation game (\cref{example:constellation_game}), arranged in order of best to worst for the yellow-colored node. Equilibrium points were found for each configuration, for each of 50,000 random instances of the game. The average relative cost reduction (for the yellow node) compared to the Nash equilibrium (the edgeless configuration) and associated 95\% confidence interval are given for each network configuration.}
\label{fig:four_node_example}
\end{figure}
As indicated in the above table, the constellation game is comprised of four MP nodes, indexed $1$ through $4$. The decision index sets for these MPs are given by 
\begin{equation*}
   J^1 := \{1,2\}, \ \ J^2 := \{3,4\}, \ \ J^3 := \{5,6\}, \ \ J^4 := \{7,8\},
\end{equation*}
and as such the private variables $\x^i$ are simply given by the named variables $\p^i$. The feasible set for each of the nodes is the set of variables $\x$ such that their private variables $\p^i$ lie within a two-dimensional box. 

The cost function for each node is a function of the parameters $\mathbf{g}^i$ and $\mathbf{r}^{i,j}$, which are referred to, respectively, as the \emph{target locations} for $\p^i$, and the \emph{target relationships} for the pair ($\p^i$, $\p^j$). A depiction of the cost function for each of the four nodes can be seen in \cref{fig:constellations} for some given value of the target and relationship vectors.  

The children for each node (and therefore the network configuration) are left unspecified, since in the analysis below, many instances of this game will be considered, each with a different set of edges between the node. Specifically, every unique network configuration will be compared. Each of these configurations correspond to selecting a set of network edges from the super-set of all possible directed edges (excluding self edges):
 \begin{equation*}
    E^{\mathrm{super}} := \left\{ (i,j) : i \in [4], j \in [4], i \ne j \right\}
\end{equation*}
 There are 12 elements in $E^\mathrm{super}$, and therefore $2^{12}=4096$ ways to construct a network configuration for a four-node MPN in this manner. However, for this analysis, only acyclic networks are considered, since cyclic networks result in degenerate MPNs as discussed in \cref{sec:formulation}. Furthermore, some sets of edges can contain redundant edges, meaning the set of reachable transitions $\mathbf{R}$ for a network is unchanged if those edges are removed. Edge sets with redundant edges are also not considered in this analysis. 

Due to symmetry of the nodes in this game, many of the remaining MPNs can also be removed from consideration. From the perspective of any particular node, the other three nodes are interchangeable under a permutation of decision indices. For example, from the perspective of node $1$, the networks defined by edge sets $\{(1,2),(2,3)\}$ and $\{(1,4),(4,2)\}$ are identical (until the goal and relationship vectors are realized). Only network architectures in which node $1$ is uniquely oriented with respect to its interchangable peers are included. 

After removing the cyclic and redundant configurations as defined, there are only $47$ remaining configurations to be analyzed. These network configurations are shown in \cref{fig:four_node_example}, with the yellow-colored node indicating the location of node $1$. As stated, the impact that various network architectures have on the  yellow node can be used to understand the impact on all other nodes in the constellation game as well via a symmetry argument.

Having these 47 unique networks identified, a randomized analysis of the benefit that each configuration provides for a given node was performed. To accomplish this, multiple instances of the constellation game were randomly generated by sampling the parameters $\mathbf{g}^i$ and $\mathbf{r}^{i,j}$ from the standard unit multivariate normal distribution:
\begin{equation*}
   \mathbf{g}^{i} \sim \mathcal{N}(\mathbf{0}, \mathbf{I}), \ \  \mathbf{r}^{i,j} \sim \mathcal{N}(\mathbf{0}, \mathbf{I}).
\end{equation*}
For each instance, and for each of the 47 network configurations, an equilibrium was computed using \cref{alg:eq_search}, and the cost incurred for node 1 at this equilibrium point was logged. Aggregating over 50,000 random game instances, the average cost reduction compared to the Nash equilibrium (the empty edge set) for each of the other configurations was computed, along with the standard error, from which $95\%$ confidence intervals on the average cost estimate could be generated. These cost estimates are displayed for each of the network configurations in \cref{fig:four_node_example}. 

The resulting ordering of network configurations, from most to least advantageous for a given node, is rather fascinating. The Nash configuration with no edges is the worst (all other configurations provide an advantage in comparison). The best configuration, with $6.24\%$ reduction in average cost relative to the Nash configuration is the four-level network with node $1$ as the single source node. Interestingly, a four-level network with node $1$ appearing further down the hierarchy is still significantly better than many other configurations in which node $1$ is a source node, sometimes even the only source node. For example, the edge configuration $\{(2,3),(3,4),(4,1)\}$ (7th best) has a $5.19\%$ cost reduction, while the configuration $\{(1,2),(1,3),(1,4)\}$ (11th best) has a $4.64\%$ reduction. 

One conclusion that could be made from these results is that a hierarchical network configuration is advantageous for all nodes in an MPN, even for the nodes at the bottom of the hierarchy. Seemingly it is better to be at the bottom of a strongly hierarchical configuration than it is to be at the top of a configuration in which a clear hierarchy is not established. This is at least true for the constellation game example explored here. Understanding the role that network configuration has on all MPNs is a topic that should be explored in later work. 

\subsection{Robust Polyhedral Avoidance}

\vspace{8pt}
\noindent\begin{minipage}\textwidth
\begin{example}[Robust Polyhedral Avoidance] \label{example:robust_avoid}
    \renewcommand{\arraystretch}{1.6}
    \begin{center}
    \begin{tabular}{ |c c|  }
     \hline
     Dec. vars: & $\mathbf{x} := \left[ \mathbf{p}^e, \  \mathbf{u}^e, \ \mathbf{p}^{o_1}, \ \mathbf{u}^{o_1}, \ \mathbf{q}^{1}, \ \epsilon^1, \ \hdots, \ \mathbf{p}^{o_M}, \ \mathbf{u}^{o_M},\  \mathbf{q}^{M},\ \epsilon^M \right]$ \\[2pt]
     Network: &
         \begin{tikzpicture}[scale=0.75,baseline=(current bounding box.center), every node/.style={circle, draw, minimum size=1.0cm,font=\scriptsize, line width=1.0}]
            \node[circle, draw] (1) at (0.0,0.0) {1};
            \node[circle, draw] (2) at (-2.0,-1.6) {2};
            \node[draw=none] (ldots) at (-1.0, -1.6) {...};
            \node[circle, draw] (i) at (0.0,-1.6) {i};
            \node[draw=none] (rdots) at (1.0, -1.6) {...};
            \node[circle, draw] (M+1) at (2.0,-1.6) {M+1};
            \node[circle, draw] (M+2) at (-2.0,-3.2) {M+2};
            \node[draw=none] (ldots2) at (-1.0, -3.2) {...};
            \node[circle, draw] (j) at (0.0,-3.2) {j};
            \node[draw=none] (rdots2) at (1.0, -3.2) {...};
            \node[circle, draw] (2M+1) at (2.0,-3.2) {2M+1};
            \node[draw=none] (space) at (0.0,-3.5) {};
            \graph[edges={line width=1pt}] {(1) -> (2); (1) -> (i); (1) -> (M+1); (2) -> (M+2); (i) -> (j); (M+1) -> (2M+1)}; 
        \end{tikzpicture}
    \\[5pt] \hline \hline
     \multicolumn{2}{|c|}{ Node $1$: } \\\hline
     Cost: & $f(\mathbf{p}^e, \mathbf{u}^e)$ \\
     Feasible set: & $\left\{ \begin{aligned} \x: \ &\epsilon^i \ge 0 \ \forall i \in [M] \\ &\mathbf{u}^e \in U^e \end{aligned} \right\}$ \\
     Private vars: & $\mathbf{u}^e$ \\
     Child nodes: & $\{2,...,M+1\}$\\
     \hline
     \multicolumn{2}{|c|}{ Nodes $i \in \{2,... ,M+1\}$: } \\\hline
     Cost: & $\epsilon^{i-1}$ \\
     Feasible set: & $\left\{\x: \mathbf{u}^{o_{i-1}} \in U^{o_{i-1}} \right\}$ \\
     Private vars: & $ \mathbf{u}^{o_{i-1}}$\\
     Child nodes: &$\{i+M\}$\\
     \hline
     \multicolumn{2}{|c|}{ Nodes $j \in \{M+2,...,2M+1\}$: } \\\hline
     Cost: & $\epsilon^{j-M-1}$ \\
     Feasible set: & $\left\{ \begin{aligned} 
     \x : \ &\mathbf{y}^e = \mathbf{p}^e + \mathbf{u}^e + \mathbf{q}^{k}, \\ 
            &\mathbf{y}^k = \mathbf{p}^{o_{k}} + \mathbf{u}^{o_{k}} + \mathbf{q}^{k},\\
            &\mathbf{A}^e \mathbf{y}^e + \mathbf{b}^e + \mathbf{1}\epsilon^{k}  \ge \mathbf{0}, \\ 
            &\mathbf{A}^{o_{k}} \mathbf{y}^{k} + \mathbf{b}^{o_{k}} + \mathbf{1}\epsilon^{k}  \ge \mathbf{0}, \\
            & (k = j-M-1)
     \end{aligned} \right\}$ \\
     Private vars: &  $\left[ \epsilon^{j-M-1}, \mathbf{q}^{j-M-1} \right] $ \\
     Child nodes: &$\{\}$\\
     \hline 
    \end{tabular}
    \end{center}
\end{example}
\vspace{\parskip}
\end{minipage}

\begin{figure}[ht]
    \centering
    \includegraphics[width=0.95\linewidth]{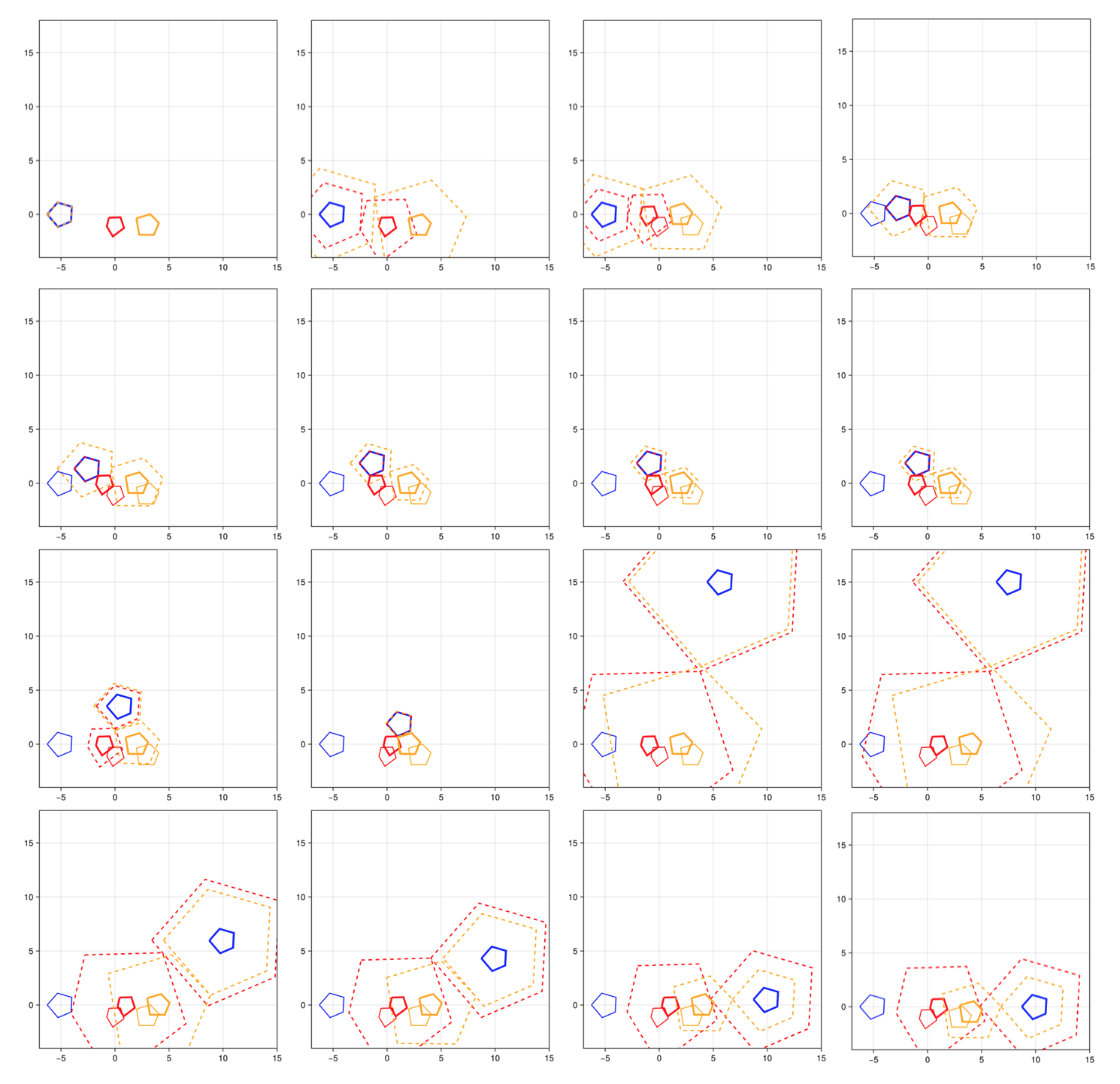}
    \caption{Visualization of each iterate of \cref{alg:eq_search} when computing an equilibrium point for \cref{example:robust_avoid}. }
    \label{fig:robust_avoid_trace}
\end{figure}

\Cref{example:robust_avoid} represents the problem of optimizing the delta ($\mathbf{u}^e$) from the initial position ($\mathbf{p}^e$) of a polygon (defined by points $\mathbf{q}$ such that $\mathbf{A}^e(\mathbf{p}^e + \mathbf{u}^e + \mathbf{q}) + \mathbf{b}^e \ge \mathbf{0}$) such that some quadratic cost $f$ is minimized, while avoiding collision with $M$ other polygons (defined by initial positions $\mathbf{p}^{o_i}$, deltas $\mathbf{u}^{o_i}$, and halfspaces given by $\mathbf{A}^{o_i}$ and $\mathbf{b}^{o_i}$) assuming the deltas of the other polygons are chosen in an adversarial manner. In this QPN, node $1$ represents the high-level optimization, nodes $i \in \{2,...,M+1\}$ represent the adversary players (one for each other polygon), and nodes $j \in \{M+2,...,2M+1\}$ represent nodes which compute the minimal expansion of two polygons so that their intersection is non-empty. Iff the expansion ($\epsilon^i$) is positive, the two polygons are not colliding. The decision variables $\mathbf{p}^e$ and $\mathbf{p}^{o_i}, i \in \{1,...,M\}$ do not appear as private variables for any of the nodes in the QPN, meaning these variables serve as parameters or ``inputs'' to the problem.  

In \cref{fig:robust_avoid_trace}, every value of $\x$ encountered when computing an equilibrium point via \cref{alg:eq_search} is displayed for a planar instance of \cref{example:robust_avoid} with two adversarial polygons (red and orange), and a quadratic cost $f$ which encourages a positive displacement for the primary polygon (blue) along the horizontal axis, and penalizes deviations from the origin along the vertical axis. Note that because this instance occurs in a planar environment, all of the variables $\mathbf{p}^e$, $\mathbf{u}^e$, $\mathbf{p}^{o_i}$, $\mathbf{u}^{o_i}$, $\mathbf{q}^i$, $i \in \{1,...,M\}$ are vectors in $\R{2}$. The set of feasible deltas are given by $\|\mathbf{u}^e\|_\infty \le 15$, and $\|\mathbf{u}^{o_i}\|_\infty \le 1$. The initial values of the variables are given by $\mathbf{p}^e = [-5,0]$, $\mathbf{p}^{o_1} = [0,-1]$, $\mathbf{p}^{o,2}=[3,-1]$, with all other variables initialized to zero. The expansions of the primary (blue) and adversarial polygons are plotted using dashed lines, and the initial configurations are shown in thin solid lines. 

As can be seen, the algorithm begins by resolving the values of $\mathbf{q}^i$ and $\epsilon^i$ for the initial values of all higher-level decision variables. The second update resolves the values of the adversarial deltas so as to maximize the intersection with the primary polygon. All remaining iterations incrementally update the value of the primary polygon delta (along with the adversarial deltas and expansion values) so as to reduce the cost function $f$ while satisfying the solution graph constraints for the lower level nodes. 

\section{Conclusion}
\label{sec:conclusion}
The concept of a Mathematical Program Network was developed. MPNs offer a framework for modeling interactions between multiple decision-makers in a manner which enables easy rearranging of the information structure or depth of reasoning of each decision process possesses. Several key results were developed to support algorithms for computing equilibrium points to MPNs, and in particular, Quadratic Program Networks. Some example networks and analyses on their solutions were presented.

The algorithms presented for solving QPNs are not without limitations. The examples presented in \cref{sec:comp_examples} were chosen so as to demonstrate the ability to solve interesting problems, while remaining tractable. When generalizing to larger problems, there are three main issues which occur. The first is that the convexity restriction for every node in the QPN is convex is limiting. Many interesting problems can be cast as QPNs involving bi-linear relationships between the decision variables of nodes at different depths in the network. However, these bi-linear relationships often result in non-convexity when paired with the solution graph constraints. Second, the numerical conditioning and floating-point errors become an issue for large-scale networks. Tolerances must be used to check equality conditions, and tuning these tolerances can be challenging. Finally, when solving for equilibria of deep QPNs, the path of iterates may encounter points where the solution graphs of the nodes contain immense numbers of polyhedral components. This problem is not fundamental---often times the solution graphs of the nodes in such networks are simple around the actual equilibrium points. However, avoiding these problematic intermediate points is difficult.

Future work will address these shortcomings, and will generalize the methods to the MPN setting, rather than focusing only on QPNs. Furthermore, broadening the framework to account for partial and imperfect information games will be an intriguing direction to consider. 

\bibliographystyle{spmpsci}      
\bibliography{main}   

%
%

\end{document}